\documentclass[a4paper,11pt,reqno]{amsart}
\usepackage{CJKutf8}

\usepackage{tikz}
\usetikzlibrary{decorations.markings}

\usepackage{amsmath,amsthm,amssymb,amsfonts,amsbsy}
\usepackage{mathtools}      % improves amsmath
\usepackage{mathabx}        % makes symbols (as $\leq$) more beautiful; loaded after mathtools 
\usepackage{enumitem,color,graphicx,cite}
\usepackage[all,pdf]{xy}

\usepackage[dvipsnames]{xcolor}

\usepackage{geometry}
\usepackage[backref=page,hyperindex=true,CJKbookmarks=true,
colorlinks,linkcolor=blue,anchorcolor=red,citecolor=cyan]{hyperref}

\theoremstyle{plain}
\newtheorem{theorem}{Theorem}[section]
\newtheorem{proposition}[theorem]{Proposition}
\newtheorem{lemma}[theorem]{Lemma}
\newtheorem{corollary}[theorem]{Corollary}

\theoremstyle{remark}
\newtheorem{remark}[theorem]{Remark}

\theoremstyle{definition}
\newtheorem{definition}[theorem]{Definition}

\newtheorem{claim}[theorem]{Claim}

\theoremstyle{plain}
\newtheorem{maintheorem}{Theorem}

\numberwithin{equation}{section}

\renewcommand*{\backref}[1]{}
\renewcommand*{\backrefalt}[4]{%
	\ifcase #1 (Not cited.)%
	\or        (Cited on page~#2.)%
	\else      (Cited on pages~#2.)%
	\fi}

\address[W. Li]{College of Mathematics, 
                Sichuan University, 
                Chengdu, Sichuan 610065, P.R. China}
\email{lwc@scu.edu.cn}

\address[M. Xia]{School of Mathematical Sciences,
                Dalian University of Technology,
                Dalian, Liaoning 116024, P.R. China}
\email{xiamingyang@dlut.edu.cn}

\title[Transitivity of partially hyperbolic skew products]
{On the Topological Transitivity of\\ Interval Extensions of Anosov Diffeomorphisms}

\author[W. Li \& M. Xia]{Wenchao LI and Mingyang XIA}

\date{\today}


\begin{document}
\begin{CJK}{UTF8}{gbsn}

\begin{abstract}
	\begin{sloppypar}
    We establish topological transitivity for a class of analytic partially hyperbolic diffeomorphisms on thickened nilmanifolds. 
    These systems take the form of skew products, constructed as interval extensions of Anosov diffeomorphisms on nilmanifolds. 
    Specifically, we develop a mechanism for topological transitivity 
    allowing arbitrary finite-order parabolic degeneracies in the central direction.
    \end{sloppypar}
\end{abstract}

\maketitle

\section{Introduction}\label{section: introduction}

Let $N$ be a nilmanifold and $M = N\times [0, 1]$ be the thickened nilmanifold. 
The boundary $\partial M$ is a disjoint union of $M_0 := N\times\{0\}$ and $M_1 := N\times\{1\}$. 
Let $\mathrm{Diff}_\partial(M)$ be the collection of diffeomorphisms on $M$ preserving $M_0$ and $M_1$ respectively.
In this paper, we study the following partially hyperbolic skew-product class in $\mathrm{Diff}_\partial(M)$.

\begin{definition}\label{definition: PHS}
    A partially hyperbolic skew product system $F\in\mathrm{Diff}_\partial(M)$ 
    is a boundary-preserving diffeomorphism defined by $F(x, t) = (f(x), \phi_x(t))$ 
    satisfying the conditions:
    \begin{itemize}
        \item $f: N \to N$ is an Anosov diffeomorphism with a hyperbolic splitting $TN = E^s\oplus E^u$;
        \item For each $x\in N$, $\phi_x$ is a diffeomorphism of $[0, 1]$ with $\phi_x(0) = 0$ and $\phi_x(1) = 1$;
        \item $\|Df|_{E^s(x)}\| < \phi_x'(t) < m(Df|_{E^u(x)})$ for any $(x, t)\in M$.
    \end{itemize}
    The collection of such partially hyperbolic diffeomorphisms on $M$ is denoted by 
    $\mathrm{PHS}(M)$.
\end{definition}

In this case, $F$ is partially hyperbolic, and by abuse of notation, 
we denote the partially hyperbolic splitting by $TM = E^s\oplus E^c\oplus E^u$. 
Note that $E^\sigma(x, t) = E^\sigma(x)\times\{0\}$ for $\sigma = s, u$ and $t = 0, 1$. 
As $F$ is dynamically coherent, 
we denote by $\mathcal{F}^\sigma$ the dynamical foliations for $\sigma = s, c, u, cs, cu$. 
Note that the center-stable and center-unstable foliations are given by 
$\mathcal{F}^{cs}(x, t) = \mathcal{F}^s(x)\times[0, 1]$ and $\mathcal{F}^{cu}(x, t) = \mathcal{F}^u(x)\times [0, 1]$.

Let $\mathrm{Per}(f)$ be the set of periodic points of $f$. 
For $p\in \mathrm{Per}(f)$ with period $\pi(p)$, 
let $\Phi_p$ be the return map defined by
\[F^{\pi(p)}(p, t) = (p, \Phi_p(t)),\ \forall\, t\in[0, 1].\]
It is clear that 
$\Phi_p = \phi_{f^{\pi(p) - 1}(p)} \circ \cdots \circ \phi_p \in\mathrm{Diff}_\partial([0, 1])$.
For simplicity, denote $F_0 = F|_{M_0}$ and $F_1 = F|_{M_1}$.
Then we have $\mathrm{Per}(F_0) = \mathrm{Per}(f)\times\{0\}$ and $\mathrm{Per}(F_1) = \mathrm{Per}(f)\times\{1\}$. 

Recall the notion of pole maps $\Phi\in\mathrm{Diff}_\partial([0, 1])$ on the unit interval.	
Let $\Phi: [0, 1]\to[0, 1]$ be a diffeomorphism that fixes both endpoints:
\begin{itemize}
	\item $\Phi$ is a north–south pole map if $0 < \Phi'(0) < 1 < \Phi'(1)$ and $\Phi(t) < t$ for any $t\in(0, 1)$;
    \item $\Phi$ is a south–north pole map if $0 < \Phi'(1) < 1 < \Phi'(0)$ and $\Phi(t) > t$ for any $t\in(0, 1)$.
\end{itemize}

\begin{definition}\label{definition: TSS}
    We say that $(p, 0)\in\mathrm{Per}(F_0)$ is a \emph{central topological sink} (resp. \emph{source})
    if there exists $\delta > 0$ such that $\Phi_p(t) < t$ (resp. $\Phi_p(t) > t$) for every $t\in (0, \delta)$. 
    Equivalently, $\Phi_p\in\mathrm{Diff}_\partial([0,1])$ is a topological north-south (resp. south-north) pole map near zero.
\end{definition}

    In this case, restricted to a $C^k$ unit interval map $\Phi:[0, 1]\to[0, 1]$, 
    there may be neutral behavior near zero as the expression 
    \[\Phi(t) = t\pm \alpha t^{n+1} + o(t^{n+1}), \text{ for some } \alpha > 0 \text{ and } 1\leq n\leq k,\]
    which means the fixed point zero being \emph{$n$-flat}. 
    This constitutes the primary concern of the present work beyond hyperbolicity.

    We denote by $\mathrm{Per}^-(F_0)$ (resp. $\mathrm{Per}^+(F_0)$) 
    the collection of central topological sink-type (resp. source-type) periodic points,
    and define $\mathrm{Per}^0(F_0) := \mathrm{Per}(F_0)\setminus(\mathrm{Per}^+(F_0)\cup \mathrm{Per}^-(F_0))$. 
    Similarly, we can define $\mathrm{Per}^*(F_1)$ for $* = +, -, 0$. 
    For example, a central topological source-type periodic point $(p, 1)\in\mathrm{Per}^+(F_1)$ 
    means that $\Phi_p$ is a topological north-south pole map near $1$, 
    in the sense that there exists $\delta > 0$ such that $\Phi_p(t) < t$ for every $t\in (1 - \delta, 1)$.

Then, following the concept of \emph{boundary interconnection} given in \cite{LSX26}, 
we introduce a topological-version of this mechanism by releasing its hyperbolic assumptions as follows. This indicates that systems with contracting/expanding centers
may exhibit purely arbitrarily $n$-flat contracting/expanding behavior,
rather than exponential contraction/expansion.

\begin{definition}\label{definition: WBI}
    We say that $F\in \mathrm{PHS}(M)$ is \textit{weakly boundary interconnected}, 
    if there exist periodic points $p_0\in\mathrm{Per}^-(F_0)$, $q_0\in\mathrm{Per}^+(F_0)$, $p_1\in \mathrm{Per}^+(F_1)$, and $q_1\in\mathrm{Per}^-(F_1)$, such that
    \begin{align*}
        W^s(p_0)\cap W^u(p_1)&\neq\varnothing;\\
        W^s(q_1)\cap W^u(q_0)&\neq\varnothing,
    \end{align*}
    where $W^s(\cdot)$ and $W^u(\cdot)$ are the stable and unstable sets respectively.
\end{definition}

In this case, $W^s(p_0)$ is an open submanifold of $\mathcal{F}^{cs}(p_0)$, and $W^u(p_1)$ is an open submanifold of $\mathcal{F}^{cu}(p_1)$ (see Lemma \ref{lemma: W^s and W^u}). 
The intersection $W^s(p_0)\pitchfork W^u(p_1)$ consists of central intervals. 
Similarly, $W^s(q_1)\pitchfork W^u(q_0)$ also consists of central intervals.
However, $W^s(\cdot)$ and $W^u(\cdot)$ may no longer be stable and unstable manifolds inside $M$, because the weak version of boundary interconnection releases the hyperbolic assumptions on the given periodic points. 

We mainly obtain the following result in the present paper.

\begin{maintheorem}\label{theorem: main theorem A}
    Let $F\in\mathrm{PHS}(M)$ be an analytic partially hyperbolic skew-product diffeomorphism. 
    If $F$ is weakly boundary interconnected,
    then $F$ is topologically transitive,
    and every positive iterate of $F$ is topologically transitive.
\end{maintheorem}

\begin{remark}
     The analytic regularity is only used in the central direction $\phi_x(\cdot)$.
     At the same time, the novelty of the mechanism of the weak boundary interconnection lies in involving arbitrary $n$-flat degeneration.
\end{remark}

\begin{remark}
     A diffeomorphism $F$ is called \emph{totally transitive} 
     if every positive iterate of $F$ is topologically transitive.
     Here weak boundary interconnection yields total transitivity.
     We would like to mention that
     Cheng, Gan and Shi \cite{CGS18} constructed examples of robustly transitive partially hyperbolic diffeomorphisms on \(\mathbb T^3\) 
     whose squares are not transitive and hence not totally transitive; 
     their examples exchange two disjoint nonempty open regions.
\end{remark}

As an application, recall the following Kan-type skew product $F_K\in\mathrm{PHS}(\mathbb{T}^d\times[0, 1])$
defined by
\[
F_K(x,t) := (Ax,\, \phi_x(t)): \mathbb{T}^d\times[0, 1] \to \mathbb{T}^d\times[0, 1],
\]
where $A\in\mathrm{Aut}(\mathbb{T}^d)$ is a hyperbolic toral automorphism with hyperbolic splitting $L^s\oplus L^u$, 
$\phi_x(\cdot) \in \mathrm{Diff}_\partial([0, 1])$ for each $x\in \mathbb{T}^d$, 
and $\|A|_{L^s}\| < \left| \phi_x'(t) \right| < m(A|_{L^u})$.
See \cite{Kan94}.

By Theorem \ref{theorem: main theorem A}, under the analytic regularity, 
we indeed establish topological transitivity for 
weakest-version constructions of Kan-type skew products (cf. \cite{BDV05,Xia23,LSX26}),
provided that the weak boundary interconnection property holds:
there exist two periodic points $p,q$ of $A\in\mathrm{Aut}(\mathbb{T}^d)$
such that $\Phi_p, \Phi_q\in\mathrm{Diff}_\partial([0, 1])$, and for each $t\in(0,1)$,
\[
\Phi_p(t)<t
\quad\text{and}\quad
\Phi_q(t)>t,
\]
that is, $\Phi_p$ and $\Phi_q$ are topological north-south and south-north pole maps, respectively.

{\color{black}
The criterion of weak boundary interconnection imposes the weakest central conditions 
    in the present formulation of Kan-type skew-product diffeomorphisms.
Here, we would like to point out that 
    \cite[Theorem 2.1]{Z26} claims topological mixing for Kan-type skew-product diffeomorphisms.
However, the proof of its fundamental Lemma 3.2 contains a gap:
    the asserted common-fiber conclusion (3.2) can not follow from Fubini’s Theorem;
    intermingling and covering of the two basins 
    do not by themselves imply that both basins meet a common central fiber inside the given open set. 
This leaves a gap in the proof of the claimed topological mixing.
    To the best of our knowledge, 
    the conclusion of Lemma 3.2 remains open in the setting considered therein, 
    and the topological mixing property is still unknown even under the stronger assumptions in \cite{LSX26}.}

% Compared with \cite{GS19}, it is an interesting question whether the topological mixing holds for interval extensions of Anosov diffeomorphisms.

% It may be interesting to explore the topologically mixing for Kan-type skew-product diffeomorphisms based on analyzing why the mechanism in \cite{GS19} works only for uniformly expanding base dynamics while fails for Anosov base dynamics.

% 传递性-弱化可扰动的双曲性
    Based on hyperbolicity, the powerful blender mechanism
        was introduced by Bonatti and D\'{\i}az \cite{BD96}
        to construct robustly transitive systems from perturbations.
    In a setting similar to the one considered here, 
        the robust cycle-type mechanism of boundary interconnection given in \cite[Definition 1.4]{LSX26},
        characterizing robust transitivity,
        may be viewed as  
{\color{black}a boundary analogue of the index-changing connection underlying classical heterodimensional cycles} 
        \cite{AS70,D95,BD96,DPU99,DR01,DR02}.
    In particular, D\'{\i}az and Rocha \cite{DR02} described the limit dynamics of coindex-one heterodimensional cycles
        through an iterated function system;
        more recently, Li and Turaev showed that arithmetic properties of specific moduli control the emergence of blenders near coindex-one heterodimensional cycles.
    
    The weak boundary interconnection enters 
        when the hyperbolic central description degenerates, 
        allowing the relevant central dynamics to be neutral      
        (see the arithmetic parameter right before Proposition \ref{proposition: center intersection}).
    In this paper, we mainly focus on establishing topological transitivity
        by the cycle‑type mechanism of weak boundary interconnection,
        an analogue of the hyperbolic mechanism 
        under arbitrary finite-order parabolic degeneration in the central direction, 
        which is beyond hyperbolicity and therefore need not be robust.
   
% 增厚斜积与Kan系统  
    On the other hand, boundary‑preserving skew‑product systems are of independent interest
        (see \cite{Mi85,IKS08,IS17,Ok17} for instance).
    	In particular, Gorodetski and Ilyashenko \cite{GI99} 
    	enhanced the blender method for skew-product systems
    	to show the coexistence of dense sets of periodic points of distinct indices; 
        Ilyashenko \cite{Il11} studied quasi-open sets 
    	in the space of boundary-preserving skew‑product diffeomorphisms of $\mathbb{T}^2\times[0,1]$,
        which are formed by the maps exhibiting the different types of thick Milnor attractors.
    Moreover, Kan's constructions also continue to be a source of interesting investigations 
        (see \cite{MW05,DVY16,BP18,UV18,NBRV21,RT22} for instance).
        In particular, a family of topologically transitive diffeomorphisms was constructed,
    	by inserting a blender in Kan's example on $\mathbb{T}^2\times [0,1]$ 
        and embedding into the boundaryless manifold (see \cite{CGS18}).

% \vspace{8pt}\noindent \textbf{Organization of this paper:} 
% In Section \ref{section: preliminaries}, we introduce some technical results that we need, including asymptotically rational independence of Birkhoff sums and analysis in the one-dimensional center dynamics. In Section \ref{section: WBI to TT}, we prove Theorem \ref{theorem: main theorem A}.

%%%%%%%%%%%%%%%%%%%%%%%%%%%%%%%%%%%%%%%%%%%%%%%%%%%%%%%%%%%%%%%%%%%%%%%%%%%%%%%%%%%%%%
%%%%%%%%%%%%%%%%%%%%%%%%%%%%%%%%%%%%%%%%%%%%%%%%%%%%%%%%%%%%%%%%%%%%%%%%%%%%%%%%%%%%%%

\section{Preliminaries}\label{section: preliminaries}
\subsection{Partial hyperbolicity}\label{subsection: PHS}

For a $C^2$ partially hyperbolic skew product $F\in\mathrm{PHS}(M)$, the dynamical foliations $\mathcal{F}^\sigma$ are H\"older continuous with $C^{1+}$ leaves, for $\sigma = s, c, u, cs, cu$. 

There is a well-known result on the regularity of holonomy maps in \cite{PSW97}, 
which we adopt in our setting as follows. 

\begin{lemma}[Theorem B,\cite{PSW97}]\label{lemma: C1-holonomy}
    For a $C^2$ partially hyperbolic skew product $F$, 
    the holonomy map along the stable foliation (resp. the unstable foliation) restricted to a center-stable leaf (resp. a center-unstable leaf) is a $C^1$ diffeomorphism. 
    Moreover, the derivative of the holonomy map is uniformly continuous with respect to the base points.
\end{lemma}
    
The following useful result on the estimate of central twist is essentially proved in \cite{LSX26}.
    
\begin{lemma}[Lemma 2.2,\cite{LSX26}]\label{lemma: central twist}
    Let $p\in \mathrm{Per}^-(F_0)$ be a central topological sink. 
    Fix $r\in \mathcal{F}^u(p)$ and 
    $\mathrm{Hol}_p^r: \mathcal{F}^{cs}_{loc}(p) \to \mathcal{F}^{cs}_{loc}(r)$ 
    the holonomy map along the unstable foliation. 
    For $x_0\in\mathcal{F}^{cs}_{loc}(p)$, denote $x_n = F^n(x_0)$ for any $n\in\mathbb{N}$, 
    and define
    $$
    y_n  = \mathcal{F}^s_{loc}(x_n)\pitchfork\mathcal{F}^c_{loc}(p),\ 
    z_n  = \mathrm{Hol}_p^r(x_n),\ 
    w_n  = \mathrm{Hol}_p^r(y_n),\ 
    w'_n = \mathcal{F}^s_{loc}(z_n)\pitchfork\mathcal{F}^c_{loc}(r).
    $$
    Then we have 
    $$d_c(w'_n,\ w_n) \leq Ce^{n\lambda_p},\ \forall\, n\in\mathbb{N}.$$
    Here $C\geq 1$ is a constant depending on $p$ and $r$, and
    $$\lambda^s(p) < \lambda_p := \frac{\lambda^u(p) - \lambda^c(p)}{\lambda^u(p) - \lambda^s(p)}\lambda^s(p) < \lambda^c(p) \leq 0,$$
    where $\lambda^c(p)$ is the central Lyapunov exponent of the fixed point $p$, and
    $$\ln\left\|DF|_{E^s(p)}\right\| =: \lambda^s(p) < \lambda^c(p) <  \lambda^u(p) := \ln m\left(DF|_{E^u(p)}\right).$$
\end{lemma}

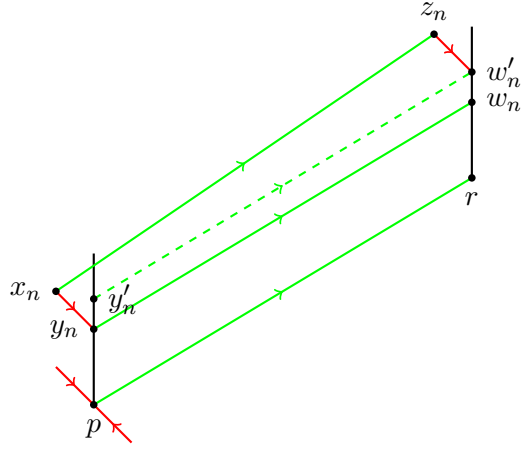
\begin{figure}[htbp]
    \centering
    \begin{tikzpicture}[mid arrow/.style = {decoration = {markings, mark = at position 0.5 with {\arrow{>}}}, postaction = {decorate}}]
        \draw [thick, mid arrow, color = red] (-0.5,0.5)--(0,0);
        \draw [thick, mid arrow, color = red] (0.5,-0.5)--(0,0);
        \draw [thick, mid arrow, color = green] (0,0)--(5,3);
        \draw [thick] (0,0)--(0,2);
        \draw [thick] (5,3)--(5,5);
        \node [circle, fill, inner sep = 1pt, label = below: $p$] at (0,0) {};
        \node [circle, fill, inner sep = 1pt, label = below: $r$] at (5,3) {};
    
        \draw [thick, mid arrow, color = red] (-0.5,1.5)--(0,1);
        \draw [thick, mid arrow, color = green] (-0.5,1.5)--(4.5,4.9);
        \draw [thick, mid arrow, color = green] (0,1)--(5,4);
        \draw [thick, mid arrow, color = red] (4.5,4.9)--(5,4.4);
        \draw [dashed, thick, mid arrow, color = green] (0,1.4)--(5,4.4);
        
        \node [circle, fill, inner sep = 1pt, label = left: $x_n$] at (-0.5,1.5) {};
        \node [circle, fill, inner sep = 1pt, label = left: $y_n$] at (0,1) {};
        \node [circle, fill, inner sep = 1pt, label = above: $z_n$] at (4.5,4.9) {};
        \node [circle, fill, inner sep = 1pt, label = right: $w_n$] at (5,4) {};
        \node [circle, fill, inner sep = 1pt, label = right: $w'_n$] at (5,4.4) {};
        \node [circle, fill, inner sep = 1pt, label = right: $y'_n$] at (0,1.4) {};
    \end{tikzpicture}
    \caption{Central twist}
    \label{figure: central twist}
\end{figure}

\begin{remark}\label{remark: uniform central twist}
    Similarly, for $q\in\mathrm{Per}^+(F_0)$, consider $F^{-1}$. We still have an estimate for the central twist, with the exponent
    \[\lambda^u(q) > \lambda_q := \frac{-\lambda^s(q) + \lambda^c(q)}{-\lambda^s(q) + \lambda^u(q)}\lambda^u(q) > \lambda^c(q) \geq 0.\]
    Denote
    \begin{align*}
        \lambda^s &:= \max\{\ln\|DF|_{E^s(x)}\|: x\in N\} < 0;\\
        \lambda^u &:= \min\{\ln m(DF|_{E^u(x)}): x\in N\} > 0.
    \end{align*}
    Then we have
    \begin{align*}
        |\lambda_p|,|\lambda_q| \geq \frac{-\lambda^u\lambda^s}{\lambda^u - \lambda^s},\ \forall\, p\in\mathrm{Per}^-(F_0),\ q\in\mathrm{Per}^+(F_0),
    \end{align*}
    indicating that the central twists have a uniformly exponential decay.
\end{remark}

\subsection{Asymptotically rational independence of Birkhoff sums}

    Here we recall some results related to Birkhoff sums. 
    This is the inductive foundation for reducing hyperbolicity from the mechanism of weak boundary interconnection.

    Recall a concept of asymptotically rational independence introduced in \cite{GS19}. 
    Given real numbers $a,b\in\mathbb{R}$, 
    define
    $$
    (a,\ b) := \inf\left\{\, |ka + lb|: k,l\in\mathbb{Z},\ ka + lb\neq 0 \,\right\}.
    $$ 
    Notice that $a,b$ are \textit{rationally independent} if and only if $(a,\ b) = 0$. 

    \begin{definition}\label{definition: ARI}
        Let $b\in\mathbb{R}$. A sequence $\{a_m\}\subseteq\mathbb{R}$ is \textit{asymptotically rationally independent} of $b$, 
        if $(a_m,\ b) \to 0$ as $m\to +\infty$.
    \end{definition}

Then recall the notions of Birkhoff sums and Birkhoff averages. 
Let $\varphi: N \to \mathbb{R}$ be an observation on a nilmanifold $N$ with $A: N \to N$ an Anosov diffeomorphism. Denote by
\begin{align*}
    &S_\varphi A(p) := \sum_{i = 0}^{\pi(p) - 1}{\varphi(A^i(p))};\\
    &\Bar{S}_\varphi A(p) := \frac{1}{\pi(p)}S_\varphi A(p),
\end{align*}
the \textit{Birkhoff sum} and the \textit{Birkhoff average} of the observable function $\varphi$ along a periodic orbit of $p\in\mathrm{Per}(A)$ with period $\pi(p)$.
    
We mention that the Birkhoff sum, driven by a return map of the skew-product system, can be regarded as an observable function involving the information of central Lyapunov exponents.
Precisely, consider $F\in\mathrm{PHS}(M)$, and define
\[\varphi(x) := \ln\|DF|_{E^c(x)}\|.\]
Then for a given $p\in \mathrm{Per}(F_0)$ with period $\pi(p)$, the Birkhoff sum $S_\varphi F_0(p)=\pi(p)\Bar{S}_\varphi F_0(p)$ is also the total central Lyapunov exponent $\pi(p)\lambda^c(p)$. 

There is the following classical \textbf{Liv\v{s}ic Theorem}, see \cite{Liv72}.
    
\begin{lemma}[Liv\v{s}ic Theorem]\label{lemma: Livsic}
    Let $A: N \to N$ be an Anosov diffeomorphism on a nilmanifold 
    and $\varphi: N \to \mathbb{R}$ be a H\"older continuous function. 
    If
    \[S_\varphi A(p)= 0,\ \forall\, p\in \mathrm{Per}(A),\]
    then there exists a H\"older continuous function $\psi: N \to \mathbb{R}$ 
    such that $\varphi= \psi - \psi\circ A$.
\end{lemma}

\begin{remark}\label{remark: Livsic}
    If $\varphi: N \to \mathbb{R}$ is a $C^k$ function, then there is also a $C^k$ solution; see \cite{dMM86}.
\end{remark}
    
Moreover, there is a rigidity result on the distribution of Birkhoff sums.

\begin{lemma}[Theorem 1.1,\cite{GSX22}]\label{lemma: dense distribution}
    Let $A: N \to N$ be an Anosov diffeomorphism on a nilmanifold and $\varphi: N \to \mathbb{R}$ be a H\"older continuous function. 
    If there exist $p_0,q_0\in \mathrm{Per}(A)$ satisfying 	
    \[S_\varphi A(p_0) < 0 < S_\varphi A(q_0),\]
    then the set $\left\{S_\varphi A(p): p\in \mathrm{Per}(A)\right\}$ is dense in $\mathbb{R}$.
\end{lemma}

\subsection{Analysis of one-dimensional dynamics}\label{subsec:one-dim}

Here we investigate the one-dimensional central dynamics, 
with a particular focus on behaviors slightly beyond the hyperbolic setting.
The following Proposition \ref{proposition: center intersection} serves as the main result of this subsection 
and provides a key generalization via the weak-version mechanism of boundary interconnection.
The proof is technical and may be skipped on a first reading.

Here we say that an orientation-preserving homeomorphism $T: (C, + \infty) \to (C', + \infty)$ 
is \emph{asymptotically} a translation by $a > 0$, if $T(t) - t \to a$ when $t\to +\infty$.
Recall that the fixed point $0$ is a topological sink 
for an orientation-preserving map $f\in\mathrm{Diff}_\partial([0, 1])$ 
if there exists $\delta>0$ such that $f(x) < x$ for each $x\in(0, \delta]$. 
This implies that $0< f'(0)\leq 1$.

\begin{lemma}\label{lemma: reparameterization}
    Let $f\in\mathrm{Diff}_\partial([0, 1])$ be a $C^1$ diffeomorphism with topological sink $0$. 
    Then there exists a reparameterization $\tau_f: (0, \delta] \to [C, +\infty)$,
    \[\tau_f(x) = \int_x^\delta\frac{du}{u - f(u)},\]
    such that $T_f := \tau_f\circ f\circ \tau_f^{-1}$ is asymptotically a translation by $t_f = -\ln f'(0)/(1 - f'(0))$. 
    Note that $t_f$ is well-defined when $0 < f'(0) < 1$ and is regarded as $1$ when $f'(0) = 1$.
\end{lemma}

\begin{proof}
    Since $0$ is a topological sink for $f$, 
    $\tau_f: (0, \delta] \to [C, +\infty)$ is a $C^2$ diffeomorphism. 
    Note that $T_f(t) > t$, for each $t\in [C, +\infty)$.

    Denote $\alpha = f'(0)\in(0,1]$. We have the expansion
    \[f(x) = \alpha x - R_f(x),\ R_f(x) = o(x)\ (x\to 0^+).\]
    For any $\varepsilon > 0$, there exists $x_\varepsilon \in(0, \delta)$ such that
    \[
    |R_f(x)| < \varepsilon x 
    \quad \text{and}\quad 
    |R_f'(x)| < \varepsilon,\ \forall\, x\in(0, x_\varepsilon).
    \]
    
    Assume that $0 < \alpha < 1$. Take $\varepsilon < \min\{\alpha, 1 - \alpha\}$, and then we have
    \begin{align*}
        T_f(t) - t &= \tau_f\circ f\circ \tau_f^{-1}(t) - \tau_f\circ\tau_f^{-1}(t)\\
        &= \int_{f\circ\tau_f^{-1}(t)}^{\tau_f^{-1}(t)}\frac{du}{u - f(u)}\\
        &= \int_{\alpha x - R_f(x)}^x\frac{du}{(1 - \alpha)u + R_f(u)}\quad \big(x = \tau_f^{-1}(t)\big)\\
        &= \int_{\alpha x - R_f(x)}^x\frac{du}{(1 - \alpha)u} - \int_{\alpha x - R_f(x)}^x\frac{R_f(u)du}{(1 - \alpha)u[(1 - \alpha)u + R_f(u)]}\\
        &= \frac{1}{1 - \alpha}\ln\frac{x}{\alpha x - R_f(x)} - \int_{\alpha x - R_f(x)}^x\frac{R_f(u)du}{(1 - \alpha)u[(1 - \alpha)u + R_f(u)]}.
    \end{align*}
    Therefore, for $t\in (\tau_f(x_\varepsilon), +\infty)$, or equivalently, $x\in (0, x_\varepsilon)$, we have
    \begin{align*}
        \left|T_f(t) - t - \frac{-\ln\alpha}{1 - \alpha}\right|
        & = \left|\frac{1}{1 - \alpha}\ln\frac{\alpha x}{\alpha x - R_f(x)} - \int_{\alpha x - R_f(x)}^x\frac{R_f(u)du}{(1 - \alpha)u[(1 - \alpha)u + R_f(u)]}\right|\\
        & \leq \frac{1}{1 - \alpha}\ln\frac{\alpha}{\alpha - \varepsilon} + \int_{\alpha x - R_f(x)}^x\frac{\varepsilon udu}{(1 - \alpha)u(1 - \alpha - \varepsilon)u}\\
        & \leq \frac{1}{1 - \alpha}\ln\frac{\alpha}{\alpha - \varepsilon} + \frac{\varepsilon}{(1 - \alpha)(1 - \alpha - \varepsilon)}\ln\frac{1}{\alpha - \varepsilon} \triangleq W_f(\varepsilon).
    \end{align*}
    Since $W_f(0^+) = 0$, we conclude that $T_f$ is asymptotically a translation by $-\ln\alpha/(1 - \alpha)$.

    Now assume that $\alpha = 1$. A similar argument shows that
    \begin{align*}
        |T_f(t) - t - 1| &\leq \int_{x - R_f(x)}^x\left|\frac{1}{R_f(u)} - \frac{1}{R_f(x)}\right|du\\
        & \leq \max_{u\in[x - R_f(x), x]}\left|\frac{R_f(x)}{R_f(u)} - 1\right|\\
        & \leq \max_{u\in[x - R_f(x), x]}\left|\frac{R_f(x)}{R_f(x) - R_f'(\xi)(x - u)} - 1\right|\quad \big(\xi = \xi(u) \in [u, x]\big)\\
        & \leq \frac{1}{1 - \varepsilon} - 1 = \frac{\varepsilon}{1 - \varepsilon}.
    \end{align*}
    It follows that $T_f$ is asymptotically a translation by $1$.
\end{proof}

Following the assumptions of Lemma \ref{lemma: reparameterization}, 
for any fixed $x\in (0, \delta]$, the interval $(f(x), x]$ is called a fundamental domain of $f$. 
Similarly, $[t, T_f(t))$ is called a fundamental domain of $T_f$. 
There is a natural projection $\pi_f^t : [t, +\infty) \to \mathbb{S}^1$ by identifying $s \sim T_f(s)$ for each $s\in[t, +\infty)$. 
Here, the metric on $\mathbb{S}^1$ is normalized to ensure that $\pi_f^t: [t, T_f(t)) \to \mathbb{S}^1$ is locally affine.

\begin{lemma}\label{lemma: asymptotical translation}
    Let $f\in\mathrm{Diff}_\partial([0, 1])$ be a real-analytic diffeomorphism with topological sink $0$. 
    Then there exists a constant $\rho_f \geq 1$ such that 
    for any closed interval $K$ lying in the closure of some fundamental domain $\Omega$ of $T_f$, 
    we have
    \[
    \rho_f^{-1}\frac{|K|}{|\Omega|} 
    \leq \frac{|T_f^k(K)|}{|T_f^k(\Omega)|} 
    \leq \rho_f\frac{|K|}{|\Omega|},\ \forall\, k\in\mathbb{N}.
    \]
\end{lemma}

\begin{proof}
    Assume that $\Omega = [a, b)$ and $K = [c, d]$, with $a \leq c < d \leq b$. We have
    \begin{align*}
        \frac{|T_f^k(K)|}{|T_f^k(\Omega)|} & = \frac{T_f^k(d) - T_f^k(c)}{T_f^k(b) - T_f^k(a)}\\
        & = \frac{d - c}{b - a}\cdot\frac{(T_f^k)'(\xi_1)}{(T_f^k)'(\xi_2)}\quad \big(\xi_1, \xi_2\in[a, b]\big)\\
        & = \frac{|K|}{|\Omega|}\cdot \exp\sum_{j = 0}^{k - 1}\ln\frac{T_f'(T_f^j(\xi_1))}{T_f'(T_f^j(\xi_2))}.
    \end{align*}
    Take $x_i = \tau_f^{-1}(\xi_i)$, $i = 1$, $2$. 
    Then $x_1$ and $x_2$ lie in the same fundamental domain of $f$, and
    \begin{align*}
        \left|\sum_{j = 0}^{k - 1}\ln\frac{T_f'(T_f^j(\xi_1))}{T_f'(T_f^j(\xi_2))}\right| & = \left|\sum_{j = 0}^{k - 1}\ln\frac{T_f'\circ\tau_f\circ f^j(x_1)}{T_f'\circ \tau_f\circ f^j(x_2)}\right|\\
        & \leq \sup_{x\in(0, \delta)}|(\ln T_f'\circ \tau_f)'(x)|\cdot \sum_{j = 0}^{k - 1}|f^j(x_1) - f^j(x_2)|\\
        & \leq \sup_{x\in(0, \delta)}|(\ln T_f'\circ \tau_f)'(x)|.
    \end{align*}
    It remains to show that $(\ln T_f'\circ\tau_f)'$ is bounded near zero. 
    Recall that $T_f = \tau_f\circ f\circ \tau_f^{-1}$. 
    It follows that $(T_f\circ\tau_f)' = (\tau_f\circ f)'$, 
    and hence $(T_f'\circ\tau_f)\cdot \tau_f' = (\tau_f'\circ f)\cdot f'$. Therefore,
    \begin{align*}
        (\ln T_f'\circ\tau_f)' &= \left(\ln \frac{\tau_f'\circ f}{\tau_f'} + \ln f'\right)'\\
        &= \left(\ln \frac{x - f(x)}{f(x) - f^2(x)}\right)' + \frac{f''}{f'}.
    \end{align*}
    Since $f'(0)\in(0,1]$, we have that $f''/f'$ is bounded near zero. Denote $w(x) = x - f(x)$, and the rest is written as
    \begin{align*}
        \left(\ln\frac{w}{w\circ f}\right)' &= \frac{w\circ f}{w}\cdot \frac{w'\cdot(w\circ f) - w\cdot(w'\circ f)\cdot f'}{w^2\circ f}\\
        &= \frac{w'\cdot(w\circ f) - w\cdot(w'\circ f)\cdot f'}{w\cdot(w\circ f)}.
    \end{align*}
    
    If $f'(0) = \alpha \in (0, 1)$, then $w(x) = (1 - \alpha)x + O(x^2)$. It follows that
    \[\left(\ln\frac{w}{w\circ f}\right)' = \frac{O(x^2)}{\alpha(1 - \alpha)^2x^2 + O(x^3)}\]
    is bounded near zero.
    
    If $f'(0) = 1$, then $w(x) = \gamma x^m + O(x^{m + 1})$ for some $\gamma > 0$ and $m \geq 2$. It follows that
    \[\left(\ln\frac{w}{w\circ f}\right)' = \frac{O(x^{2m})}{\gamma^2x^{2m} + O(x^{2m + 1})}\]
    is also bounded near zero.
\end{proof}

\begin{corollary}\label{corollary: uniformly bounded}
    There exists a constant $C=C_f \geq 1$, such that for any small closed interval $K$ in $[t, +\infty)$, 
    the projection $\pi_f^t(K)$ is a closed interval satisfying
    \[C^{-1}|K| \leq |\pi_f^t(K)| \leq C|K|.\]
\end{corollary}

\begin{proof}
    Recall that $T_f$ is asymptotically a translation by $t_f$. 
    It follows that $|T_f^k(\Omega)| \to t_f$ for any fundamental domain $\Omega$ of $T_f$. 
    In particular, choose $\Omega = [t, T_f(t))$, and we know that
    \[
    0 < \omega^- := 
    \inf_{k\in\mathbb{N}}|T_f^k(\Omega)|\leq \sup_{k\in\mathbb{N}}|T_f^k(\Omega)|
    =:\omega^+  < +\infty.
    \] 
    When $K$ is sufficiently small so that $K$ does not contain any fundamental domain of $T_f$, there exist at most two intervals $K_1$, $K_2$ in $\Omega$ and $k_1$, $k_2\in\mathbb{N}$ 
    such that $K = T_f^{k_1}(K_1)\sqcup T_f^{k_2}(K_2)$. 
    Lemma \ref{lemma: asymptotical translation} implies that
    \[|\pi_f^t(K)| = \frac{|K_1|}{|\Omega|} + \frac{|K_2|}{|\Omega|} \in \left[\frac{|K|}{\rho_f\omega^+},\frac{\rho_f|K|}{\omega^-}\right].\]
    This completes the proof.
\end{proof}

If there are two $C^1$ diffeomorphisms $f$, $g\in\mathrm{Diff}_\partial([0, 1])$ admitting topological sink $0$, 
then we define the following asymptotic arithmetic parameter
\[[f/g] := \lim_{x\to 0^+}\frac{t_f\cdot(x - f(x))}{t_g\cdot(x - g(x))},\]
whenever the limit exists and takes a value in $[0, +\infty)$.
This parameter coincides with the logarithmic multiplier ratio appearing in \cite{LT24} under hyperbolic assumptions.

For an analytic diffeomorphism $f\in\mathrm{Diff}_\partial([0,1])$, 
we denote by $\mathrm{ord}(f)$ the order of $f$ which is defined as follows:
If $f'(0)\neq 1$, let $\mathrm{ord}(f) = 1$; if $f \equiv \mathrm{Id}$, let $\mathrm{ord}(f) = +\infty$; 
otherwise there exists $m \geq 2$ such that
\[f(x) = x + \frac{1}{m!}f^{(m)}(0)x^m + o(x^m),\]
and we let $\mathrm{ord}(f) = m$.

\begin{proposition}\label{proposition: center intersection}
    Let $f$, $g\in\mathrm{Diff}_\partial([0, 1])$ be real-analytic diffeomorphisms with topological sink $0$, 
    and let $h \in \mathrm{Diff}_\partial([0, 1])$. 
    If $g'(0) \neq 1$ or $h'(0) = 1$, then there exist $\varepsilon_0 > 0$ and an increasing function $L(\varepsilon): (0, \varepsilon_0) \to (0, +\infty)$ 
    depending on $g$, satisfying $L(0^+) = 0$, such that the following conclusion holds: 
    If $[f/g] = 0$ or $([f/g],\ 1) < \varepsilon$, 
    then for any two closed intervals $I$, $J\subseteq (0, \delta]$ with $|J| > L(\varepsilon)$, 
    there are two strictly increasing subsequences $\{k_n\}$, $\{l_n\}$ of $\mathbb{N}$, and a constant $C > 0$, such that
    \[
    |h\circ f^{k_n}(I)\cap g^{l_n}(J)| \geq C \min\{e^{\sigma_f(k_n)}, e^{\sigma_g(l_n)}\},
    \ \forall\, n\in\mathbb{N},
    \]
    where
    \[
    \sigma_f(n) := \left\{\begin{array}{ll}
         n\ln f'(0), & \mathrm{ord}(f) = 1; \\[1ex]
         -\frac{\mathrm{ord}(f)}{\mathrm{ord}(f) - 1}\ln n, & \mathrm{ord}(f) \geq 2,
    \end{array}\right.
    \]
    and
    \[\sup_{n\in\mathbb{N}}\left|\frac{\sigma_f(k_n)}{\mathrm{ord}(f)} - \frac{\sigma_g(l_n)}{\mathrm{ord}(g)}\right| < +\infty.\]
\end{proposition}

\begin{proof}
    Given a single point $x\in(0, 1)$ and a closed interval $J\subseteq (0, 1)$, 
    in order to find $k$, $l\in\mathbb{N}$ such that $h\circ f^k(x)\in g^l(J)$, 
    it suffices to show that $H\circ T_f^k(t)\in T_g^l(K)$, 
    where $H = \tau_g\circ h\circ \tau_f^{-1}$, $t = \tau_f(x)$, and $K = \tau_g(J)$. 
    See the following diagram.
    \[
    \xymatrix{[C_f, +\infty)\ar[d]_{T_f} 
    & (0, \delta_f] \ar[l]_{\ \ \tau_f} \ar[r]^{h}\ar[d]_f 
    & (0, \delta_g] \ar[r]^{\tau_g\ \ } \ar[d]^{g} 
    & [C_g, +\infty)\ar[d]^{T_g}\\ [C_f, +\infty) 
    & (0, \delta_f] \ar[l]^{\ \ \tau_f} \ar[r]_h
    & (0, \delta_g] \ar[r]_{\tau_g\ \ } 
    & [C_g, +\infty)}
    \]
    
    Notice that
    \[
    H'(t) 
    = \tau_g'(h(x))h'(x)(\tau_f^{-1})'(t) 
    = h'(x)\frac{\tau_g'(h(x))}{\tau_f'(x)} 
    = h'(x)\frac{x - f(x)}{h(x) - g(h(x))}.
    \]
    If $g'(0) \neq 1$, we have the expansion $g(x) = g'(0)x - R_g(x)$, where $R_g(x) = o(x)\ (x\to 0^+)$, and hence
    \[
    \lim_{t\to+\infty} H'(t) 
    = h'(0)\lim_{x\to 0^+}\frac{x - f(x)}{x - g(x)}\lim_{x\to 0^+}\frac{(1 - g'(0))x + R_g(x)}{(1 - g'(0))h(x) + R_g(h(x))} 
    = \frac{t_g}{t_f}[f/g].
    \]
    If $g'(0) = 1$ and $h'(0) = 1$, then $R_g(x) = \beta x^m + o(x^m)$ for some $\beta > 0$ and $m \geq 2$, hence
    \[
    \lim_{t\to +\infty} H'(t) 
    = h'(0)\lim_{x\to 0^+}\frac{x - f(x)}{x - g(x)}\lim_{x\to 0^+}\frac{R_g(x)}{R_g(h(x))} 
    = \frac{t_g}{t_f}[f/g].
    \]
    In both cases, we can define
    \begin{align*}
        \widetilde{H}   & = t_g^{-1}\circ H\circ t_f,\\
        \widetilde{T}_f & = t_f^{-1}\circ T_f,\\
        \widetilde{T}_g & = t_g^{-1}\circ T_g.
    \end{align*}
    Then $$\widetilde{H}'(+\infty):=\lim_{t\to + \infty}\widetilde{H}'(t) = [f/g],$$ 
    and $\widetilde{T}_f,\widetilde{T}_g$ are asymptotically translations by $1$. 
    For the desired intersection, it now suffices to find $k$, $l\in\mathbb{N}$ such that 
    $\widetilde{H}(\widetilde{T}_f^k(t)) \in \widetilde{T}_g^l(K)$.

    Fix any $t_0$ and consider the projection $\pi_g^{t_0}: [t_0, +\infty) \to \mathbb{S}^1$.
    We first have the following.

    \begin{claim}\label{claim: density}
        There exist $\varepsilon_0 > 0$ and an increasing function $\widetilde{L}(\varepsilon): (0, \varepsilon_0) \to (0, +\infty)$ 
        such that the following holds:
        \begin{itemize}
            \item If $[f/g] = 0$, then $\Lambda_t := \{\pi_g^{t_0}(\widetilde{H}(\widetilde{T}_f^k(t))): k\in\mathbb{N}\}$ is dense in $\mathbb{S}^1$, for any $t \geq t_0$.
            \item If $([f/g], 1) < \varepsilon$, then there exists $k_0\in\mathbb{Z}^+$ depending on $[f/g]$ and $\varepsilon$, such that for any $t$ sufficiently large, $\Lambda_t^{k_0} :=\{\pi_g^{t_0}(\widetilde{H}(\widetilde{T}_f^k(t))): 0\leq k \leq k_0\}$ is $\widetilde{L}(\varepsilon)$-dense in $\mathbb{S}^1$.
        \end{itemize}
    \end{claim}
    \begin{proof}[Proof of Claim \ref{claim: density}]
        Take the fundamental domain $\Omega = [t_0, T_g(t_0))$, and denote $\Omega_j = T_g^j(\Omega)$. 
        
        First we consider the case $[f/g] = 0$, and show that the set $\Lambda_t$ is dense in $\mathbb{S}^1$. 
        Assume for contradiction that there exists an open interval $U$ in $\mathbb{S}^1$ such that $U\cap \Lambda_t = \varnothing$. 
        Then there exists an open interval $U_j$ in each $\Omega_j$ such that 
        $\widetilde{H}(\widetilde{T}_f^k(t))\not\in U_j$ for every $k\in \mathbb{N}$. 
        Moreover, according to Corollary \ref{corollary: uniformly bounded}, 
        we can choose $\inf_j|U_j| > 0$. 
        It then contradicts the fact that $\widetilde{T}^k(t) \to +\infty$ as $k\to+\infty$, $\widetilde{H}(+\infty) = +\infty$, 
        and
        \[
        |\widetilde{H}(\widetilde{T}_f^{k + 1}(t)) - \widetilde{H}(\widetilde{T}_f^k(t))| 
        \leq \widetilde{H}'(\xi_t^k)|\widetilde{T}^{k + 1}(t) - \widetilde{T}^k(t)| 
        \to 0\ (k \to +\infty).
        \]

        Then we consider the case $([f/g], 1) < \varepsilon$. 
        Notice that there exist $k_0$, $l_0\in\mathbb{Z}^+$ depending on $[f/g]$ and $\varepsilon$, 
        such that for any $x\in \mathbb{R}$, the set
        \[\{x + k[f/g] - l: 0\leq k \leq k_0, 0\leq l \leq l_0\}\]
        is $\varepsilon$-dense in $[x, x + 1]$. 
        In particular, since $\widetilde{T}_f$ and $\widetilde{T}_g$ are asymptotically translations by $1$, and $\widetilde{H}'(+\infty) = [f/g]$, we can take $x_0$ sufficiently large such that, 
        for any $x > x_0$,
        \begin{align*}
        &\max
        \left\{
        \sup_{0\leq k \leq k_0} \big\{|\widetilde{T}_f^k(x) - x - k|\big\}, \, 
        \sup_{0\leq l \leq l_0} \big\{|\widetilde{T}_g^{-l}(x) - x + l|\big\}, \,\,
        \big|\widetilde{H}'(x) - [f/g]\big| \,
        \right\}\\
        &\leq \frac{\varepsilon}{(1 + k_0 + [f/g])}.
        \end{align*}
        It follows that for $t$ sufficiently large,
        \[
        \sup_{\substack{0\leq k \leq k_0\\ 0\leq l \leq l_0}}|\widetilde{T}_g^{-l}(\widetilde{H}(\widetilde{T}_f^k(t))) - (\widetilde{H}(t) + k[f/g] - l)|
        \leq \frac{(1 + k_0 + [f/g])\varepsilon}{(1 + k_0 + [f/g])} = \varepsilon.
        \]
        Therefore,
        \[\{\widetilde{T}_g^{-l}(\widetilde{H}(\widetilde{T}_f^k(t))): 0\leq k \leq k_0, 0\leq l \leq l_0\}\]
        is $2\varepsilon$-dense in the fundamental domain $[\widetilde{H}(t), \widetilde{T}_g(\widetilde{H}(t))]$, 
        for any $t$ sufficiently large. 
        Finally, by Corollary \ref{corollary: uniformly bounded}, $\Lambda_t^{k_0}$ is $\widetilde{L}(\varepsilon)$-dense in $\mathbb{S}^1$, 
        where $\widetilde{L}(\varepsilon) = 2C_g\varepsilon$. 
        
        This completes the proof of Claim \ref{claim: density}.
    \end{proof}
     
    For any closed interval $K$ with $|K| \geq C_g\widetilde{L}(\varepsilon)$, 
    by Corollary \ref{corollary: uniformly bounded}, $|\pi_g^{t_0}(K)|\geq \widetilde{L}(\varepsilon)$. By Claim \ref{claim: density}, if $[f/g] = 0$, for any $t \geq t_0$, we have infinitely many $k\in\mathbb{N}$ such that $\pi_g^{t_0}(\widetilde{H}(\widetilde{T}_f^k(t))) \in \pi_g^{t_0}(K)$. Take a positive and strictly increasing subsequence $\{k_n\}$, and then there exists $l_n\in\mathbb{Z}$ such that $\widetilde{H}(\widetilde{T}_f^{k_n}(t)) \in \widetilde{T}_g^{l_n}(K)$. Since $\widetilde{T}_f^{k_n}(t)\to +\infty$ and $\widetilde{H}(+\infty) = +\infty$, by taking a subsequence again, we can assume that $\{l_n\}$ is also positive and increasing.

    On the other hand, if $([f/g], 1) < \varepsilon$, by Claim \ref{claim: density}, for any $t$ sufficiently large, 
    there exists $0\leq k \leq k_0$ such that $\pi_g^{t_0}(\widetilde{H}(\widetilde{T}_f^k(t))) \in \pi_g^{t_0}(K)$. 
    Equivalently, there exists $l\in\mathbb{Z}$ such that $\widetilde{H}(\widetilde{T}_f^k(t)) \in \widetilde{T}_g^l(K)$.
    For any $t$ and $K$, to obtain the sequences $\{k_n\}$ and $\{l_n\}$, 
    apply the above argument to $\widetilde{T}_f^n(t)$ and $K$, 
    since $\widetilde{T}_f^n(t)$ is sufficiently large as $n\to+\infty$. 
    We then have $\widetilde{H}(\widetilde{T}_f^{k'_n + n}(t)) \in \widetilde{T}_g^{l_n}(K)$ 
    for some $0\leq k'_n\leq k_0$ and $l_n\in\mathbb{Z}$. 
    Obviously $k_n := k'_n + n \to +\infty$, 
    and hence $l_n \to +\infty$. 
    By taking a subsequence if necessary, 
    we can require that the sequences $\{k_n\}$ and $\{l_n\}$ are positive and strictly increasing.

    Return to the discussion in $(0, \delta_g]$. Define
    \[L(\varepsilon) \triangleq \sup\{|\tau_g^{-1}(K)|: |K| = C_g\widetilde{L}(\varepsilon)\}.\]
    We conclude that when $[f/g] = 0$ or $([f/g], 1) < \varepsilon$, for any $x\in (0, \delta_f]$ and any closed interval $J\subseteq (0, \delta_g]$ with $|J| > L(\varepsilon)$, there exist strictly increasing subsequences $\{k_n\}$ and $\{l_n\}$ such that $h\circ f^{k_n}(x) \in g^{l_n}(J)$.
    
    Finally, we consider for two closed intervals $I$ and $J$. 
    Assume that $I = [a,b]$ and $J = [c,d]$, and take the middle third $J' = [c', d']$ of $J$. 
    By replacing $L(\varepsilon)$ by $3L(\varepsilon)$, 
    we apply the above argument to a single point $x\in I$ and the closed interval $J'$, 
    and conclude that there exist $\{k_n\}$ and $\{l_n\}$ such that $h\circ f^{k_n}(I)\cap g^{l_n}(J')\neq \varnothing$. 
    
    We then have the following two cases.
    
\noindent \textbf{Case 1: $h\circ f^{k_n}(I)$ is totally contained in $g^{l_n}(J)$.} 
    In this case we have 
    \[|h\circ f^{k_n}(I)\cap g^{l_n}(J)| = |h\circ f^{k_n}(I)|.\] 
    
    If $f(x) = e^{-\alpha} x + O(x^2)$ for some $\alpha > 0$, then we claim that there exists a constant $C > 0$ such that
    \[|h\circ f^{k_n}(I)|\sim Ce^{-k_n\alpha}.\]
    To see this, denote $x_n = f^n(x)$, and then $\sum_{i = 0}^{+\infty}x_i$ converges uniformly w.r.t. $x\in I$. Therefore,
    \[x_n = x\prod_{i = 0}^{n - 1}\frac{f^{i + 1}(x)}{f^i(x)} = xe^{-n\alpha}\prod_{i = 0}^{n - 1}(1 + O(x_i))\sim Ce^{-n\alpha}, \exists\, C > 0.\]
    Notice that
    \[|h\circ f^{k_n}(I)| = \int_a^b (h\circ f^{k_n})'(x)\,dx = \int_a^b h'(x_{k_n})\prod_{i = 0}^{k_n - 1}f'(x_i)\,dx = e^{-k_n\alpha}\cdot C_n,\]
    where
    \[C_n := \int_a^b h'(x_{k_n})\prod_{i = 0}^{k_n - 1}\frac{f'(x_i)}{e^{-\alpha}}\,dx = \int_a^bh'(x_{k_n})\prod_{i = 0}^{k_n - 1}(1 + O(x_i))\,dx\]
    converges to some constant $C > 0$.

    If $f(x) = x - \alpha x^m + O(x^{m + 1})$ for some $\alpha > 0$ and $m \geq 2$, then we claim that there exists a constant $C > 0$ such that
    \[|h\circ f^{k_n}(I)| \sim Ck_n^{-\frac{m}{m - 1}}.\]
    To see this, denote $x_n = f^n(x)$, $y_n = x_n^{-(m - 1)}$, and we have
    \begin{align*}
        &x_{n + 1} = x_n(1 - \alpha x_n^{m - 1} + o(x_n^{m - 1}));\\
        &y_{n + 1} = y_n(1 - \alpha y_n^{-1} + o(y_n^{-1}))^{-(m - 1)}.
    \end{align*}
    As a result,
    \[y_{n + 1} = y_n(1 + (m - 1)\alpha y_n^{-1} + o(y_n^{-1})),\]
    and hence
    \[y_{n + 1} - y_n \to (m - 1)\alpha, \text{ and } \frac{1}{n}y_n \to (m - 1)\alpha.\]
    Consequently,
    \[x_n \sim \left(\frac{1}{(m - 1)\alpha n}\right)^{\frac{1}{m - 1}}.\]
    Now we consider
    \begin{align*}
        |h\circ f^{k_n}(I)| &= \int_a^b h'(x_{k_n})f'(x)\prod_{i = 1}^{k_n - 1}f'(x_i)\,dx\\
        &= \int_a^b h'(x_{k_n})f'(x)\prod_{i = 1}^{k_n - 1}(1 - m\alpha x_i^{m - 1} + O(x_i^m))\,dx = {k_n}^{-\frac{m}{m - 1}}\cdot C_n,
    \end{align*}
    where
    \begin{align*}
        C_n &= k_n^{\frac{m}{m - 1}}\int_a^bh'(x_{k_n})f'(x)\exp\sum_{i = 1}^{k_n - 1}\ln\big(1 - m\alpha x_i^{m - 1} + O(x_i^m)\big)\,dx\\
        &= k_n^{\frac{m}{m - 1}}\int_a^b h'(x_{k_n})f'(x)\exp\sum_{i = 1}^{k_n - 1}\ln\left(1 - \frac{m}{m - 1}\frac{1}{i} + O(i^{-\frac{m}{m - 1}})\right)dx\\
        &= k_n^{\frac{m}{m - 1}}\int_a^b h'(x_{k_n})f'(x)\exp\sum_{i = 1}^{k_n - 1}\left(- \frac{m}{m - 1}\frac{1}{i} + O(i^{-\frac{m}{m - 1}})\right)dx\\
        &= k_n^{\frac{m}{m - 1}}\int_a^b h'(x_{k_n})f'(x)\exp\left(-\frac{m}{m - 1}\ln k_n + \mathrm{const.} + o(1)\right)dx
    \end{align*}
    converges to some constant $C > 0$.
    
    Summarizing the above discussion for \textbf{Case 1}, we have
    \[|h\circ f^{k_n}(I)\cap g^{l_n}(J)| \geq Ce^{\sigma_f(k_n)},\]
    for some constant $C > 0$.
    
\noindent \textbf{Case 2: $h\circ f^{k_n}(I)$ is not totally contained in $g^{l_n}(J)$.} 
    In this case, $h\circ f^{k_n}(I)$ must contain one of the connected components of $g^{l_n}(J)\setminus g^{l_n}(J')$: $g^{l_n}([c, c'])$ or $g^{l_n}([d',d])$. By a similar argument, we have
    \[|h\circ f^{k_n}(I)\cap g^{l_n}(J)| \geq Ce^{\sigma_g(l_n)}.\]
    
    Combining \textbf{Case 1} and \textbf{Case 2}, we obtain the final estimate:
    \[|h\circ f^{k_n}(I)\cap g^{l_n}(J)| \geq C\min\{e^{\sigma_f(k_n)}, e^{\sigma_g(l_n)}\}.\]
    Moreover, we consider the relationship between $\sigma_f(k_n)$ and $\sigma_g(l_n)$.
    
    Notice that there exists $C \geq 1$ such that for any $x$, $y \in (0, \delta_g]$ lying in the closure of the same fundamental domain of $g$, we have $C^{-1} \leq x/y \leq C$. In fact, assume that $x$, $y\in [g(z), z]$, and we have
    \[\frac{g(z)}{z}\leq\frac{x}{y}\leq \frac{z}{g(z)}.\]
    Since $g(z)/z \to g'(0) \in(0, 1]$, it suffices to take $C = \sup\{z/g(z): z\in(0, \delta_g]\}$.

    Now we have $h\circ f^{k_n}(x) \in g^{l_n}(J)$ for some $x\in I$. Therefore,
    \[h\circ f^{k_n}(x) \asymp g^{l_n}(d).\]
    Denote $x_n = f^{k_n}(x)$ and $y_n = g^{l_n}(d)$. By the above argument, we have
    \[
    \begin{alignedat}{2}
        & x_n \asymp \exp \frac{\sigma_f(k_n)}{\mathrm{ord}(f)}; \qquad & h(x_n) \asymp x_n;\\
        & y_n \asymp \exp \frac{\sigma_g(l_n)}{\mathrm{ord}(g)}; \qquad & h(x_n) \asymp y_n.
    \end{alignedat}
    \]
    Finally, we conclude that there exists a constant $C > 0$ such that
    \[\left|\frac{\sigma_f(k_n)}{\mathrm{ord}(f)} - \frac{\sigma_g(l_n)}{\mathrm{ord}(g)}\right| \leq C.\]
  
    This completes the proof of Proposition \ref{proposition: center intersection}.
\end{proof}

\section{Weak Boundary Interconnection Implies Topological Transitivity}\label{section: WBI to TT}

In this section, we establish topological transitivity for partially hyperbolic skew products 
with the mechanism of weak boundary interconnection. $F$ is assumed to be analytic.

First, we characterize the stable set of a topological sink.
Note that a subset of $M$ is called \textit{$\sigma$-saturated} if it is a union of $\sigma$-leaves, for $\sigma=s, u$.

\begin{lemma}\label{lemma: W^s and W^u}
    Let $p\in\mathrm{Per}^-(F_0)$ be a central topological sink with period $\pi(p)$. 
    Then there exists $0 < s_p \leq 1$ satisfying
    \[
    I_p^c := W^s(p)\cap\mathcal{F}^c(p) = [0, s_p) 
    \quad\text{and}\quad 
    W^s(p) = \bigcup_{x\in I_p^c}\mathcal{F}^s(x).
    \]
\end{lemma}

\begin{proof}
    Recall that $F^{\pi(p)}(p, t) = (p, \Phi_p(t))$. 
    Therefore, $\Phi_p \in\mathrm{Diff}_\partial([0, 1])$. 
    Since $p\in\mathrm{Per}^-(F_0)$, $\Phi_p$ has a minimal positive fixed point 
    $s_p = \sup\{s\in(0, 1]: \Phi_p(t) < t, \ \forall \, 0<t<s\}$. 
    It follows that $\Phi_p(t) < t$ for each $ t\in(0, s_p)$.

    By the above argument, $I_p^c := W^s(p)\cap\mathcal{F}^c(p) = [0, s_p)$. 
    Therefore we have
    $$W^s(p) \supseteq \bigcup_{x\in I_p^c}\mathcal{F}^s(x).$$
    Now take $y\in W^s(p)$, and then we have $\pi_0(y)\in\mathcal{F}^s(\pi_0(p))$, 
    where $\pi_0: M \to M_0$ is the natural projection. It follows that $y\in\mathcal{F}^{cs}(p)$. 
    Assume that $y\in\mathcal{F}^s(z)$ for some $z\in\mathcal{F}^c(p)$. 
    Then $z\in W^s(p)$ and hence $z\in I_p^c$, since $W^s(p)$ is $s$-saturated.
\end{proof}

The analogous conclusions hold for periodic points of both $F_0$ and $F_1$: 
the stable set (resp. unstable set) of a central topological sink-type (resp. topological source-type) periodic point 
can be well characterized by its contracting center (resp. expanding center). 
For simplicity, we omit the statements for the remaining cases.

\begin{lemma}\label{lemma: density of s and u}
    Assume that $F\in\mathrm{PHS}(M)$ is weakly boundary interconnected. Then for any $p\in\mathrm{Per}^-(F_0)\cup \mathrm{Per}^-(F_1)$, $W^s(p)$ is dense in $M$; for any $p\in \mathrm{Per}^+(F_0)\cup\mathrm{Per}^+(F_1)$, $W^u(p)$ is dense in $M$. In particular, for any $p_0\in\mathrm{Per}^-(F_0)$, $q_0\in\mathrm{Per}^+(F_0)$, $p_1\in\mathrm{Per}^+(F_1)$, $q_1\in\mathrm{Per}^-(F_1)$, we have
    \begin{align*}
        W^s(p_0)\pitchfork W^u(p_1)&\neq\varnothing;\\
        W^s(q_1)\pitchfork W^u(q_0)&\neq\varnothing.
    \end{align*}
\end{lemma}

\begin{proof}
    Since \(F\) is weakly boundary interconnected, there are periodic points $p_0\in\mathrm{Per}^-(F_0)$, $q_0\in\mathrm{Per}^+(F_0)$, $p_1\in\mathrm{Per}^+(F_1)$, and $q_1\in\mathrm{Per}^-(F_1)$, such that the intersection property holds.
    
    For convenience, we define $\widehat{p}_0 := \pi_1(p_0)$ and $\widehat{p}_1 := \pi_0(p_1)$, 
    where $\pi_i : M \to  M_i$ is the natural projection for $i = 0, 1$. 
    Recall that $W^s(p_0)\pitchfork W^u(p_1)$ contains a central interval. 
    If the interval lies in a central leaf with endpoints $z_0\in M_0$ and $z_1\in M_1$, 
    then there exists an open central interval $J \subseteq \mathcal{F}^c(z_0)$ with an endpoint $z_0$, 
    lying in $W^s(p_0)$ and containing an interval in $W^u(p_1)$.

    Then consider $F^{-N}(J)$. One of its endpoints is $F_0^{-N}(z_0)$, which tends to $\widehat{p}_1$; 
    the other endpoint tends to $p_1$, since it lies in $W^u(p_1)$. 
    Moreover, $F^{-N}(J)$ is a central interval, hence $F^{-N}(J)$ tends to $\mathcal{F}^c(p_1)$, 
    in the sense that
    $$
    \forall\, x\in\mathcal{F}^c(p_1),\ \forall\, \varepsilon > 0,\ \exists\, N\in\mathbb{N},\ \exists\, x'\in F^{-N}(J),
    \ \text{such that}\ d(x',\ x) < \varepsilon.
    $$

    Since $F_0$ is an Anosov diffeomorphism on a nilmanifold, 
    the stable foliation restricted to $M_0$ is minimal. 
    It follows that $\mathcal{F}^{cs}$ is minimal, 
    because $\mathcal{F}^{cs}(\cdot) = \pi_0^{-1}(\mathcal{F}^s(\pi_0(\cdot)))$.

    Now $\mathcal{F}^{cs}(p_1)$ is dense in $M$. 
    For any non-empty open subset $U\subseteq M$, 
    there is $x\in\mathcal{F}^c(p_1)$ such that $\mathcal{F}^s(x)\cap U\neq\varnothing$. 
    Note that
    $$
    \mathcal{F}^s(x) = \bigcup_{k\in\mathbb{N}} F^{-k}(\mathcal{F}^s_\varepsilon(F^k(x))),
    \ \forall\, \varepsilon > 0.
    $$
    Therefore, there exists $K = K(\varepsilon)\in\mathbb{N}$ 
    such that $\mathcal{F}^s_\varepsilon(F^K(x))\cap F^K(U)\neq\varnothing$. 
    Recall that $F^{-N}(J)$ tends to $\mathcal{F}^c(p_1)$. 
    Then there are $N\in\mathbb{N}$ and $x'\in F^{-N}(J)$ 
    such that $\mathcal{F}^s_\varepsilon(F^K(x'))\cap F^K(U)\neq\varnothing$. 
    Now we have that $\mathcal{F}^s(F^N(F^K(x')))\cap F^{K + N}(U)\neq\varnothing$. 
    Since $F^N(F^K(x'))\in F^K(J)\subseteq W^s(p_0)$, we conclude that
    $$W^s(p_0) \cap U = F^{-(K + N)}(W^s(p_0))\cap U \neq\varnothing.$$
    This shows the density of $W^s(p_0)$. 
    By the same argument, $W^u(q_0), W^u(p_1)$ and $W^s(q_1)$ are also dense in $M$.

    Generally, for any $p\in\mathrm{Per}^-(F_0)$, 
    by Lemma \ref{lemma: W^s and W^u}, $W^s(p)$ is an open submanifold of $\mathcal{F}^{cs}(p)$. 
    Since $W^u(p_1)$ is an open submanifold of $\mathcal{F}^{cu}(p_1)$ and dense in $M$, 
    we conclude that $W^s(p)\pitchfork W^u(p_1)\neq\varnothing$. 
    Now replace $p_0$ by $p$ in the above argument; it follows that $W^s(p)$ is dense in $M$. 
    A similar conclusion holds for $p\in \mathrm{Per}^-(F_1)$, $\mathrm{Per}^+(F_0)$, and $\mathrm{Per}^+(F_1)$. 
    The density implies the intersection property,
    which ends the proof of Lemma \ref{lemma: density of s and u}.
\end{proof}

For $p\in\mathrm{Per}(F_0)$, the return map $\Phi_p\in\mathrm{Diff}_\partial([0, 1])$ is assumed to be analytic. We define the order of $p$ by $\mathrm{ord}(p) := \mathrm{ord}(\Phi_p)$.

Now we prove Theorem \ref{theorem: main theorem A}.

\begin{proof}[Proof of Theorem \ref{theorem: main theorem A}]
    To prove topological transitivity,
    it suffices to show that for any two non-empty open subsets $U,\ V\subseteq M$, 
    we have $U^*\cap V^*\neq\varnothing$, where
    $$
    U^* := \bigcup_{n\in\mathbb{N}}F^n(U) 
    \quad\text{and}\quad  
    V^* := \bigcup_{n\in\mathbb{N}}F^{-n}(V).
    $$
    Note that $F(U^*)\subseteq U^*$, $F^{-1}(V^*)\subseteq V^*$.
    
    Since $F$ is weakly boundary interconnected, by Lemma \ref{lemma: density of s and u}, for any two pairs of periodic points $p_i, q_i\in M_i\ (i = 0, 1)$ 
    such that $p_0, q_1$ are central topological sinks, $p_1, q_0$ are central topological sources, we have
    \begin{align*}
        W^s(p_0)\pitchfork W^u(p_1)&\neq\varnothing;\\
        W^s(q_1)\pitchfork W^u(q_0)&\neq\varnothing.
    \end{align*}
    
    The next goal is to find $p_0\in\mathrm{Per}^-(F_0)$ and $q_0\in\mathrm{Per}^+(F_0)$ properly, 
    such that either $[\Phi_{p_0}/\Phi_{q_0}^{-1}]$ or $[\Phi_{q_0}^{-1}/\Phi_{p_0}]$,
    defined right before Proposition \ref{proposition: center intersection},
    exists and takes a value in $[0, +\infty)$, 
    so that we can apply the analysis introduced in Subsection \ref{subsec:one-dim}.

    Let $p_0\in\mathrm{Per}^-(F_0)$ have the minimal order 
    among the central topological sinks of $F_0$, 
    and $q_0\in \mathrm{Per}^+(F_0)$ have the minimal order 
    among the central topological sources of $F_0$. 
    By the analyticity, we have
    \begin{align*}
        \Phi_{p_0}(t) &= e^{-\alpha}t + o(t), \text{ or } \Phi_{p_0}(t) = t - \alpha t^n + o(t^n);\\
        \Phi_{q_0}(t) &= e^{\beta}t + o(t),   \text{ or } \Phi_{q_0}(t) = t + \beta t^{n'} + o(t^{n'});
    \end{align*}
    where $\alpha$, $\beta > 0$ and $n,n' \geq 2$. 
    
    Without loss of generality, 
    we assume $\mathrm{ord}(p_0) \geq \mathrm{ord}(q_0)$. 
    There are the cases as follows.

\noindent \textbf{Case 1. $\mathrm{ord}(p_0) > \mathrm{ord}(q_0)$.} 
    In this case we have $[\Phi_{p_0}/\Phi_{q_0}^{-1}] = 0$.

\noindent \textbf{Case 2. $\mathrm{ord}(p_0) = \mathrm{ord}(q_0) = 1$.} 
    In this case we have
    \[
    [\Phi_{p_0}/\Phi_{q_0}^{-1}] 
    = \lim_{t \to 0^+}\frac{\frac{\alpha}{1 - e^{-\alpha}}(t - e^{-\alpha}t + o(t))}{\frac{\beta}{1 - e^{-\beta}}(t - e^{-\beta}t + o(t))} 
    = \lim_{t \to 0^+}\frac{\alpha t + o(t)}{\beta t + o(t)}
    = \frac{\alpha}{\beta}.
    \]

\noindent \textbf{Case 3. $\mathrm{ord}(p_0) = \mathrm{ord}(q_0) = n > 1$.} 
    In this case we have
    \[
    [\Phi_{p_0}/\Phi_{q_0}^{-1}] 
    = \lim_{t \to 0^+}\frac{\alpha t^n + o(t^n)}{\beta t^n + o(t^n)} 
    = \frac{\alpha}{\beta}.
    \]

    Therefore, $[\Phi_{p_0}/\Phi^{-1}_{q_0}]$ exists, as long as we take the periodic points with the minimal order.
    
    For technical reasons, the hyperbolicity in the central direction of $q_0\in\mathrm{Per}^+(F_0)$ should be sufficiently weak, in the sense that
    \[\lambda^c(q_0) < \frac{\mathrm{ord}(q_0)}{\mathrm{ord}(p_0)}\lambda_{q_0}.\]
    Recall that $\lambda_{q_0} > 0$ describes the decay of the central twist; see Lemma \ref{lemma: central twist}.

    In \textbf{Case 2} and \textbf{Case 3}, $\mathrm{ord}(p_0) = \mathrm{ord}(q_0)$, and the property $\lambda^c(q_0) < \lambda_{q_0}$ holds automatically. 
    In \textbf{Case 1}, if $\mathrm{ord}(q_0) > 1$, then $\lambda^c(q_0) = 0$ and the property also holds. 
    Therefore, it remains to deal with the case $\mathrm{ord}(p_0) > \mathrm{ord}(q_0) = 1$.

    In this case, $q_0$ can be chosen to ensure that $\lambda^c(q_0)$ is arbitrarily close to zero. 

    \begin{claim}\label{claim: weak hyperbolicity}
        If $\mathrm{ord}(p_0) > \mathrm{ord}(q_0) = 1$, then we can modify $q_0$ without changing its order, such that
        \[\lambda^c(q_0) < \frac{\mathrm{ord}(q_0)}{\mathrm{ord}(p_0)}\frac{-\lambda^s\lambda^u}{\lambda^u - \lambda^s}.\]
        In particular,
        \[\lambda^c(q_0) < \frac{\mathrm{ord}(q_0)}{\mathrm{ord}(p_0)}\lambda_{q_0}.\]
    \end{claim}

    \begin{proof}[Proof of Claim \ref{claim: weak hyperbolicity}]
        For simplicity, we assume that $p_0,q_0$ are fixed points.
        
        Recall that $F(x, t) = (f(x), \phi_x(t))$. Define $\varphi: M_0 \to \mathbb{R}$ by
        \[\varphi(x) := \ln\phi_x'(0).\]
        Then $\varphi$ is H\"older continuous and there exist constants $C_\varphi \geq 1$ and $0 < \theta < 1$ such that
        \[|\varphi(x) - \varphi(y)| \leq C_\varphi d(x, y)^\theta,\ \forall\, x, y\in M_0.\]
        Recall that we denote the Birkhoff sum for $p\in\mathrm{Per}(F_0)$ with period $\pi(p)$ by
        \[S_\varphi F_0(p) := \sum_{i = 0}^{\pi(p) - 1}\varphi(F_0^i(p)).\]
        We then have
        \(S_\varphi F_0(p) = \pi(p)\lambda^c(p)\) and \(\bar{S}_\varphi F_0(p) = \lambda^c(p).\)
        Recall that $\lambda^c(p_0) = 0<\lambda^c(q_0)$. 
        
        We modify $q_0$ by shadowing, 
        see Figure \ref{figure: periodic pseudo orbit for weaker hyperbolicity}.

        As an Anosov diffeomorphism on the nilmanifold $M_0$, $F_0$ is topologically transitive. 
        Hence $p_0$ and $q_0$ are homoclinically related, and there are points $x_0, y_0\in M_0$ with
        \[
        x_0 \in \mathcal{F}^u(p_0) \pitchfork\mathcal{F}^s(q_0) 
        \quad\text{and}\quad 
        y_0 \in \mathcal{F}^u(q_0) \pitchfork \mathcal{F}^s(p_0).
        \]
        
        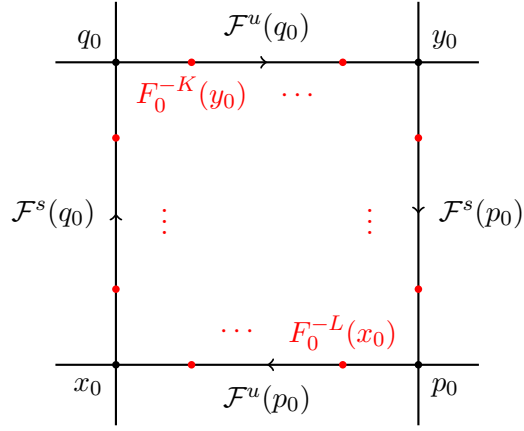
\begin{figure}[htbp]
			\centering
            \begin{tikzpicture}[scale=0.8, mid arrow/.style = {decoration = {markings, mark = at position 0.5 with {\arrow{>}}}, postaction = {decorate}}]
                \draw [thick, mid arrow](-1, 5)--(6, 5);
                \draw [thick, mid arrow](5, 6)--(5, -1);
                \draw [thick, mid arrow](6, 0)--(-1, 0);
                \draw [thick, mid arrow](0, -1)--(0, 6);
                \node [circle, fill, inner sep = 1pt, label = south east: $p_0$] at (5,0){};
                \node [circle, fill, inner sep = 1pt, label = south west: $x_0$] at (0,0){};
                \node [circle, fill, inner sep = 1pt, label = north west: $q_0$] at (0,5){};
                \node [circle, fill, inner sep = 1pt, label = north east: $y_0$] at (5,5){};
                \node [label = above: $\mathcal{F}^u(q_0)$] at (2.5, 5) {};
                \node [label = right: $\mathcal{F}^s(p_0)$] at (5, 2.5) {};
                \node [label = below: $\mathcal{F}^u(p_0)$] at (2.5, 0) {};
                \node [label = left: $\mathcal{F}^s(q_0)$] at (0, 2.5) {};
                \node [circle, fill = red, inner sep = 1pt, label = below: \textcolor{red}{$F_0^{-K}(y_0)$}] at (1.25, 5) {};
                \node [label = below: $\textcolor{red}{\cdots}$] at (3,4.9) {};
                \node [circle, fill = red, inner sep = 1pt] at (3.75, 5) {};
                \node [circle, fill = red, inner sep = 1pt] at (5, 3.75) {};
                \node [label = left: $\textcolor{red}{\vdots}$] at (4.6,2.5) {};
                \node [circle, fill = red, inner sep = 1pt] at (5, 1.25) {};
                \node [circle, fill = red, inner sep = 1pt, label = above: \textcolor{red}{$F_0^{-L}(x_0)$}] at (3.75, 0) {};
                \node [label = above: $\textcolor{red}{\cdots}$] at (2,0.1) {};
                \node [circle, fill = red, inner sep = 1pt] at (1.25, 0) {};
                \node [circle, fill = red, inner sep = 1pt] at (0, 1.25) {};
                \node [label = right: $\textcolor{red}{\vdots}$] at (0.4,2.5) {};
                \node [circle, fill = red, inner sep = 1pt] at (0, 3.75) {};
            \end{tikzpicture}
			\caption{Periodic pseudo-orbit for weaker hyperbolicity}
			\label{figure: periodic pseudo orbit for weaker hyperbolicity}
		\end{figure}
        
        We then consider the following pseudo orbit segment:
        \[\mathcal{P} := \left(F_0^{-K}(y_0), \cdots, y_0, \cdots, F_0^{L - 1}(y_0),\ F_0^{-L}(x_0), \cdots, x_0, \cdots, F_0^{K - 1}(x_0)\right).\]
        Here, $K, L\in\mathbb{N}$ are large numbers to be decided. 
        The sum of $\varphi$ over $\mathcal{P}$ is
        $$
        S_\varphi(\mathcal{P}):= \sum_{x\in\mathcal{P}}\varphi(x) =\sum_{i = -K}^{L - 1}\varphi(F_0^i(y_0)) + \sum_{i = -L}^{K - 1}\varphi(F_0^i(x_0)).
        $$
        Since $F_0$ is uniformly contracting (resp. expanding) in stable (resp. unstable) manifolds and $\varphi$ is H\"older continuous, 
        we have the following convergent series:
        \begin{align*}
			P&=
            \sum\limits_{i = -\infty}^{-1} \left[\varphi(F_0^i(y_0)) - \varphi(q_0)\right] + \sum_{i = 0}^{+\infty} \left[\varphi(F_0^i(y_0)) - \varphi(p_0)\right]\\
			&+\sum_{i = -\infty}^{-1} \left[\varphi(F_0^i(x_0)) - \varphi(p_0)\right] + \sum_{i = 0}^{+\infty} \left[\varphi(F_0^i(x_0)) - \varphi(q_0)\right].
		\end{align*}
        As a result, given $\varepsilon > 0$, 
        there exist $K^0,L^0\in\mathbb{N}$, such that 
        for $K \geq K^0$ and $L \geq L^0$, 
        \begin{align}
            \left|S_\varphi(\mathcal{P}) - \left(P + 2K \lambda^c(q_0) + 2L\lambda^c(p_0)\right)\right| < \varepsilon. \label{formula: sum over pseudo orbit, q_0}
        \end{align}
        On the other hand, by the classical \textbf{exponential Anosov shadowing lemma}, 
        there are constants $0 < \delta_0 < 1$, $C_0 \geq 1$ and $0 < \mu_0 < 1$ such that for every $0 < \delta < \delta_0$, 
        every $\delta$-pseudo orbit $\{x_i: i_A\leq i \leq i_B\}\ (i_A \leq 0 \leq i_B)$ is exponentially shadowed by a periodic point $p$, 
        in the sense that
        $$d(F_0^i(p),\ x_i) < C_0\mu_0^{|i|}\delta,\ \forall\, i_A \leq i \leq i_B.$$
        As a result, there exist $K^1, L^1\in\mathbb{N}$ such that 
        when $K \geq K^1$ and $L \geq L^1$, 
        $\mathcal{P}$ is a $\delta_0\varepsilon$-pseudo orbit, 
        exponentially shadowed by a periodic point $q'_0\in\mathrm{Per}(F_0)$. 
        In fact, it suffices to take $K^1$ and $L^1$ so large that
        \begin{align*}
            &\max\left\{d(F_0^{L^1}(y_0),\ p_0),\ d(F_0^{-L^1}(x_0),\ p_0)\right\} < \frac{1}{2}\delta_0\varepsilon;\\
            &\max\left\{d(F_0^{K^1}(x_0),\ q_0),\ d(F_0^{-K^1}(y_0),\ q_0)\right\} < \frac{1}{2}\delta_0\varepsilon.
        \end{align*}
        
        Now we consider $S_\varphi F_0(q'_0)$, which is the sum of $\varphi$ over the orbit of $q'_0$. 
        Since $\varphi$ is H\"older continuous, 
        the difference between $S_\varphi F_0(q'_0)$ and $S_\varphi(\mathcal{P})$ is controlled by
        \begin{align}
            \left|S_\varphi F_0(q'_0) - S_\varphi(\mathcal{P})\right| < C_\varphi C_0^\theta\frac{2\delta_0^\theta}{1 - \mu_0^\theta}\varepsilon^\theta. \label{formula: sum over periodic orbit, q_0}
        \end{align}
        
        Combining (\ref{formula: sum over pseudo orbit, q_0}) and (\ref{formula: sum over periodic orbit, q_0}), there exists a constant $C \geq 1$ such that
        \begin{align}
            \left|S_\varphi F_0(q'_0) - \left(P + 2K\lambda^c(q_0) + 2L\lambda^c(p_0)\right)\right| < C\varepsilon^\theta. \label{formula: estimate of q'_0}
        \end{align}
        Here we can choose $K > \max\{K^0,\, K^1\}$ and $L > \max\{L^0,\, L^1\}$, so that the Birkhoff average is also controlled:
        \begin{align*}
            \lambda^c(q'_0) &= \bar{S}_\varphi F_0(q'_0) = \frac{S_\varphi F_0(q'_0)}{2K + 2L}\\
            &> \frac{2K\lambda^c(q_0) - |P| - C\varepsilon^\theta}{2K + 2L};
        \end{align*}
        \begin{align*}
            \lambda^c(q'_0) &= \bar{S}_\varphi F_0(q'_0) = \frac{S_\varphi F_0(q'_0)}{2K + 2L}\\
            & < \frac{2K\lambda^c(q_0) + |P| + C\varepsilon^\theta}{2K + 2L}.
        \end{align*}
        As long as $K$ is sufficiently large, we have $\lambda^c(q'_0) > 0$; as long as $L = K^2$ is sufficiently large, we have that $\lambda^c(q'_0)$ is arbitrarily small.

        This completes the proof of Claim \ref{claim: weak hyperbolicity}.
    \end{proof}

    We now fix the point \(q_0\), making the replacement in Claim \ref{claim: weak hyperbolicity} when necessary, 
    and assume that \(q_0\) is a fixed point of \(F_0\).
    
    In order to apply Proposition \ref{proposition: center intersection} for $f = \Phi_{p_0}$ and $g = \Phi^{-1}_{q_0}$, we also need $[f/g] = 0$ or $([f/g], 1) < \varepsilon$. Recall that in \textbf{Case 1}, we have $[f/g] = 0$.
    
    In \textbf{Case 2} or \textbf{Case 3}, we have $[f/g] = \alpha/\beta$.
    Therefore, we still need to modify $p_0$ without changing its order, so that $([\Phi_{p_0}/\Phi^{-1}_{q_0}], 1)$ is sufficiently small.
    
    Define $\varphi: M_0 \to \mathbb{R}$ by
    \[
    \begin{alignedat}{2}
        &\varphi(x) := \ln\phi_x'(0),               \qquad &\text{ when } n = 1;\\
        &\varphi(x) := \frac{1}{n!}\phi_x^{(n)}(0), \qquad &\text{ when } n \geq 2.
    \end{alignedat}
    \]
    As a result,
    $S_\varphi F_0(p_0) = \varphi(p_0) = -\alpha$ and $S_\varphi F_0(q_0) = \varphi(q_0) = \beta$.
    
    %We have the following Liv\v{s}ic reparameterization. 
    %Note that $\Phi_p = \mathrm{Id}$ for $p\in\mathrm{Per}^0(F_0)$, since $\Phi_p$ is analytic.
    \begin{claim}\label{claim: Livsic reparameterization}
        Assume that $\mathrm{ord}(p) \geq n \geq 2$ for every $p\in\mathrm{Per}(F_0)$. 
        Then there is a reparameterization $U(x, t) = (x, u_x(t))$ near $t = 0$, such that $G := U\circ F\circ U^{-1} = (f(x), \psi_x(t))$ has the following property: for any $x\in N$,
        $$
        \psi_x'(0) = 1; \quad
        \psi_x^{(m)}(0) = 0,\ \forall\, 2\leq m < n.
        $$
    \end{claim}
    \begin{proof}[Proof of Claim \ref{claim: Livsic reparameterization}]
        We prove the conclusion inductively. 
        
        First assume $n = 2$. 
        Then $\Phi_p'(0) = 1$ for each $p\in\mathrm{Per}(F_0)$. 
        Take $\varphi(x) = \ln\phi_x'(0)$. 
        Then the sum of $\varphi$ over every periodic orbit vanishes, 
        and hence the equation $\varphi = v\circ f - v$ has a smooth solution $v$ (see Remark \ref{remark: Livsic}). 
        Consider $U(x, t) = (x, e^{-v(x)}t)$ near $t = 0$. Then we have
        \begin{align*}
            \psi_x(t)  &= e^{-v(f(x))}\phi_x(e^{v(x)}t);\\
            \psi_x'(0) &= e^{-v(f(x)) + \varphi(x) + v(x)} = 1.
        \end{align*}
        
        Now we assume that $\mathrm{ord}(p) \geq n + 1$. 
        By the inductive hypothesis, we can assume that $\phi_x'(0) = 1$ and $\phi_x^{(m)}(0) = 0$, 
        for each $x\in N$ and $2\leq m \leq n - 1$. 
        Take $$\varphi(x) = \frac{1}{n!}\phi_x^{(n)}(0),$$ 
        and notice that the sum of $\varphi$ over a periodic orbit $\mathrm{Orb}(p)$ is exactly $\frac{1}{n!}\Phi_p^{(n)}(0)$, 
        which vanishes since $\mathrm{ord}(p) \geq n + 1$. 
        Therefore we also have a smooth solution $v$ for the equation $\varphi = v\circ f - v$ (see also Remark \ref{remark: Livsic}).
        Consider $U(x, t) = (x, t - v(x)t^n)$ near $t = 0$. 
        Then we have
        \begin{align*}
            &\psi_x(t) = t + (-v(f(x)) + \varphi(x) + v(x))t^n + o(t^n);\\
            &\psi_x^{(m)}(0) = 0,\ \forall\, 2\leq m \leq n.
        \end{align*}
        This completes the induction and finishes the proof of Claim \ref{claim: Livsic reparameterization}.
    \end{proof}
    
    As a result, under the assumption that $\mathrm{ord}(p_0) = \mathrm{ord}(q_0) = n$, if $n = 1$, then we have
    \[
    S_\varphi F_0(p) 
    = \sum_{i = 0}^{\pi(p) - 1}\varphi(F_0^i(p)) 
    = \sum_{i = 0}^{\pi(p) - 1}\ln \phi_{F_0^i(p)}'(0) 
    = \ln \Phi_p'(0);
    \]
    if $n > 1$, 
    by the choice of $n$ we can take a new coordinate near $t = 0$, 
    and conclude that $\phi_{F_0^i(p)}'(0) = 1$ and $\phi_{F_0^i(p)}^{(k)}(0)= 0$, 
    for every $0\leq i \leq \pi(p) - 1$ and $2\leq k < n$. 
    As a result,
    \[
    S_\varphi F_0(p) 
    = \sum_{i = 0}^{\pi(p) - 1}\varphi(F_0^i(p)) 
    = \frac{1}{n!}\sum_{i = 0}^{\pi(p) - 1} \phi_{F_0^i(p)}^{(n)}(0) 
    = \frac{1}{n!}\Phi_p^{(n)}(0).
    \]
    In both cases, we conclude that $S_\varphi F_0(p) < 0$ if and only if $p\in\mathrm{Per}^-(F_0)$ and $\mathrm{ord}(p) = n$; 
    and $S_\varphi F_0(p) > 0$ if and only if $p\in\mathrm{Per}^+(F_0)$ and $\mathrm{ord}(p) = n$. Moreover,
    \[-\frac{S_\varphi F_0(p)}{S_\varphi F_0(q)} = [\Phi_p/\Phi_q^{-1}]\]
    whenever $p\in\mathrm{Per}^-(F_0)$ and $q\in\mathrm{Per}^+(F_0)$ admit the minimal order $n$.
    
    Now we can deal with \textbf{Case 2} and \textbf{Case 3} where $\mathrm{ord}(p_0) = \mathrm{ord}(q_0) = n \geq 1$.
    
    \begin{claim}\label{claim: approximation}
        Fix $p_0\in\mathrm{Per}^-(F_0)$ and $q_0\in\mathrm{Per}^+(F_0)$ with the minimal order, 
        and assume that they have the same order $n$. 
        Then there exists a sequence of points $p_{0,m}\in\mathrm{Per}^-(F_0)$ with order $n$,
        such that
        %$\bar{S}_\varphi F_0(p_{0,m}) \to \bar{S}_\varphi F_0(p_0)$ and
        %$(S_\varphi F_0(p_{0,m}),\ S_\varphi F_0(q_0)) \to 0$. 
        %As a result,
        \[
        ([\Phi_{p_{0,m}}/\Phi_{q_0}^{-1}],\ 1) 
        = \left(-\frac{S_\varphi F_0(p_{0,m})}{S_\varphi F_0(q_0)},\ 1\right) 
        \to 0\ (m \to +\infty).
        \]
    \end{claim}
    \begin{proof}[Proof of Claim \ref{claim: approximation}]
        Since $S_\varphi F_0(p_0) < 0 < S_\varphi F_0(q_0)$, by Lemma \ref{lemma: dense distribution}, we take $p_{0,m}\in\mathrm{Per}(F_0)$ such that
        \[-2^{-m} < S_\varphi F_0(p_{0,m}) < 0.\]
        As a result,
        \[\left(-\frac{S_\varphi F_0(p_{0,m})}{S_\varphi F_0(q_0)},\ 1\right) \leq \left|\frac{S_\varphi F_0(p_{0,m})}{S_\varphi F_0(q_0)}\right| \to 0.\]
        This ends the proof of Claim \ref{claim: approximation}.
    \end{proof}

Now we continue to prove $U^*\cap V^* \neq\varnothing$. 
    First, we show that the points in $V^*$ appear near the expanding central interval $I_{q_0}^c$, see Figure \ref{figure: holonomy map q}.
    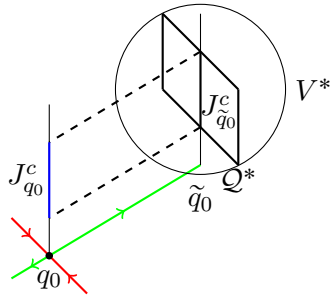
\begin{figure}[htbp]
		\centering
        \begin{tikzpicture}[mid arrow/.style = {decoration = {markings, mark = at position 0.5 with {\arrow{>}}}, postaction = {decorate}}]
            \draw [thick, mid arrow, color = red] (-0.5,0.5)--(0,0);
            \draw [thick, mid arrow, color = red] (0.5,-0.5)--(0,0);
            \draw [thick, mid arrow, color = green] (0,0)--(-0.5,-0.3);
            \draw [thick, mid arrow, color = green] (0,0)--(2,1.2);
            \node [circle, fill, inner sep = 1pt, label = below: $q_0$] at (0,0) {};
            \node [label = below: $\widetilde{q}_0$] at (2,1.2) {};
            
            \draw (0,0)--(0,2);
            \draw (2,1.2)--(2,3.2);

            \draw [thick] (2,1.7)--(2,2.7);
            \draw [thick] (1.5,2.2)--(2.5,1.2);
            \draw [thick] (1.5,3.2)--(2.5,2.2);
            \draw [thick] (1.5,2.2)--(1.5,3.2);
            \draw [thick] (2.5,1.2)--(2.5,2.2);
            \node [label = center: $J_{\widetilde{q}_0}^c$] at (2.25, 1.9) {};
            \node [label = center: $\mathcal{Q}^*$] at (2.5,1) {};
            
            \draw (2,2.2) circle (32pt);
            \node [label = center: $V^*$] at (3.5, 2.2) {};

            \draw [dashed, thick] (2,1.7)--(0,0.5);
            \draw [dashed, thick] (2,2.7)--(0,1.5);
            \draw [thick, color = blue] (0,0.5)--(0,1.5);
            \node [label = center: $J_{q_0}^c$] at (-0.3, 1) {};
        \end{tikzpicture}
        \caption{Holonomy map near $q_0$}
        \label{figure: holonomy map q}
    \end{figure}
    
    Notice that $q_0$ is a central topological source-type fixed point, 
    \[W^u(q_0) = \bigcup_{k\in\mathbb{N}}F^k\left(\mathcal{F}^{cu}_{loc}(q_0)\right).\]
    Since $W^u(q_0)$ is dense in $M$ (Lemma \ref{lemma: density of s and u})
    and $F^{-k}(V^*) \subseteq V^*$ for any $k\in\mathbb{N}$, 
    we have
    \[V^*\cap \mathcal{F}^{cu}_{loc}(q_0)\neq \varnothing.\]
    Therefore, we can take
    \[\widetilde{q}_0\in \pi_0(V^*)\cap \mathcal{F}^u_{loc}(q_0) \subseteq M_0.\]
    We can further require that $V^*\cap \mathcal{F}^c(\widetilde{q}_0)$ contains a closed central interval $J_{\widetilde{q}_0}^c$ 
    satisfying $\mathcal{Q}^* := \mathcal{F}^s_{loc}(J_{\widetilde{q}_0}^c)\subseteq V^*$. 
    Let $J_{q_0}^c\subseteq \mathcal{F}^c_{loc}(q_0)$ be the image of $J_{\widetilde{q}_0}^c$ 
    under the holonomy map along the unstable leaves, restricted to $\mathcal{F}^{cu}_{loc}(q_0)$. 
    Since $q_0$ is a central topological source, 
    we may further assume that $J_{q_0}^c\subseteq I_{q_0}^c$.

    Take $\varepsilon > 0$ sufficiently small, 
    so that $|J_{q_0}^c| > L(\varepsilon)$ as in Proposition \ref{proposition: center intersection}. 
    In the case of $\mathrm{ord}(p_0) = \mathrm{ord}(q_0)$, 
    by Claim \ref{claim: approximation}, we can take $m\in\mathbb{Z}^+$ sufficiently large, such that
    \[([\Phi_{p_{0,m}}/\Phi_{q_0}^{-1}],\ 1) < \varepsilon.\] 
    With a slight abuse of notation, for a given $\varepsilon > 0$, 
    we write $p_{0,m}$ as $p_0$ for simplicity.
    We now achieve the goal of obtaining a pair of periodic points 
    $p_0\in\mathrm{Per}^-(F_0)$ and $q_0\in\mathrm{Per}^+(F_0)$ 
    satisfying $[\Phi_{p_0}/\Phi_{q_0}^{-1}] = 0$ or $([\Phi_{p_0}/\Phi_{q_0}^{-1}],\ 1) < \varepsilon$.
    
    Similarly, since $p_0$ is a central topological sink-type periodic point, 
    $$W^s(p_0) = \bigcup_{k\in\mathbb{N}}F^{-k\pi(p_0)}\left(\mathcal{F}^{cs}_{loc}(p_0)\right).$$
    Then, by the density of $W^s(p_0)$ (Lemma \ref{lemma: density of s and u}), we have
    $$U^*\cap \mathcal{F}^{cs}_{loc}(p_0)\neq \varnothing.$$
    Therefore, we can take
    $$\widetilde{p}_0\in \pi_0(U^*)\cap \mathcal{F}^s_{loc}(p_0) \subseteq M_0.$$
    Moreover, $U^*\cap \mathcal{F}^c(\widetilde{p}_0)$ contains a closed central interval $J_{\widetilde{p}_0}^c$, satisfying $\mathcal{P}^* := \mathcal{F}^u_{loc}(J_{\widetilde{p}_0}^c)\subseteq U^*$ (see Figure \ref{figure: holonomy map p}). 
    Let $J_{p_0}^c \subseteq\mathcal{F}^c(p_0)$ be the image of $J_{\widetilde{p}_0}^c$ 
    under the holonomy map along the stable leaves, restricted to $\mathcal{F}^{cs}_{loc}(p_0)$. 
    Since $p_0$ is a central topological sink, we may further assume that $J_{p_0}^c \subseteq I_{p_0}^c$.

    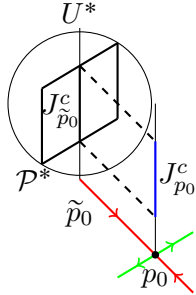
\begin{figure}[htbp]
		\centering
        \begin{tikzpicture}[mid arrow/.style = {decoration = {markings, mark = at position 0.5 with {\arrow{>}}}, postaction = {decorate}}]
            \draw [thick, mid arrow, color = red] (0,1)--(1,0);
            \draw [thick, mid arrow, color = red] (1.5,-0.5)--(1,0);
            \draw [thick, mid arrow, color = green] (1,0)--(1.5,0.3);
            \draw [thick, mid arrow, color = green] (1,0)--(0.5,-0.3);
            \node [circle, fill, inner sep = 1pt, label = below: $p_0$] at (1,0) {};

            \node [label = below: $\widetilde{p}_0$] at (0,1) {};

            \draw (0,1)--(0,3);
            \draw (1,0)--(1,2);
            
            \draw [thick] (-0.5,1.2)--(-0.5,2.2);
            \draw [thick] (0,1.5)--(0,2.5);
            \draw [thick] (0.5,1.8)--(0.5,2.8);
            \draw [thick] (-0.5,1.2)--(0.5,1.8);
            \draw [thick] (-0.5,2.2)--(0.5,2.8);
            \draw (0,2) circle (27pt);
            \node [label = center: $\mathcal{P}^*$] at (-0.6,1) {};
            \node [label = center: $U^*$] at (0,3.2) {};
            \node [label = center: $J_{\widetilde{p}_0}^c$] at (-0.25,1.9) {};

            \draw [dashed, thick] (0,1.5)--(1,0.5);
            \draw [dashed, thick] (0,2.5)--(1,1.5);
            \draw [thick, color = blue] (1,0.5)--(1,1.5);
            \node [label = center: $J_{p_0}^c$] at (1.3,1) {};
        \end{tikzpicture}
        \caption{Holonomy map near $p_0$}
        \label{figure: holonomy map p}
    \end{figure}

    Fix $r_0 \in \mathcal{F}^u(p_0) \pitchfork \mathcal{F}^s(q_0)$. 
    For simplicity, we denote by $\mathrm{Hol}_x^y$ the holonomy map from $x$ to $y$ 
    along stable or unstable manifolds defined on $cu$-leaves or $cs$-leaves; 
    see Figure \ref{figure: central intersection}.

    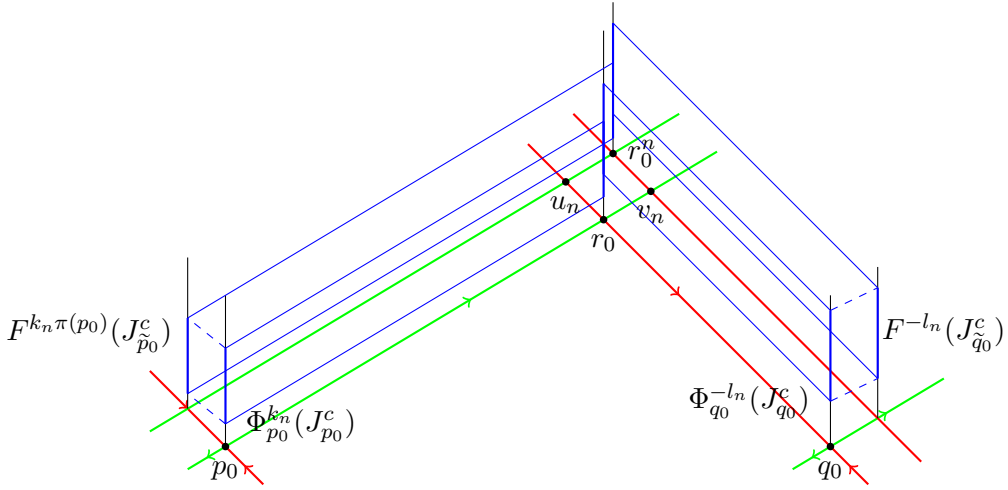
\begin{figure}[htbp]
		\centering
        \begin{tikzpicture}[mid arrow/.style = {decoration = {markings, mark = at position 0.5 with {\arrow{>}}}, postaction = {decorate}}]
            \draw [thick, mid arrow, color = green] (0,0)--(-0.5,-0.3);
            \draw [thick, mid arrow, color = green] (0,0)--(6.5,3.9);
            \draw [thick, mid arrow, color = red] (-1,1)--(0,0);
            \draw [thick, mid arrow, color = red] (0.5,-0.5)--(0,0);
            \draw (0,0)--(0,2);
            \node [circle, fill, inner sep = 1pt, label = below: $p_0$] at (0,0) {};

            \draw [thick, mid arrow, color = red] (4,4)-- (8,0);
            \draw [thick, mid arrow, color = red] (8.5,-0.5)-- (8,0);
            \draw [thick, mid arrow, color = green] (8,0)-- (7.5,-0.3);
            \draw [thick, mid arrow, color = green] (8,0)-- (9.5,0.9);
            \draw (8,0)--(8,2);
            \node [circle, fill, inner sep = 1pt, label = below: $q_0$] at (8,0) {};

            \node [circle, fill, inner sep = 1pt, label = below: $r_0$] at (5,3) {};
            \draw (5,3)--(5,5.5);

            \draw [thick, color = green] (-1,0.2)--(6,4.4);
            \draw [thick, color = red] (4.6, 4.4)--(9.2, -0.2);
            
            \draw [thick, color = blue] (0,0.3)-- (0,1.3);
            \node [label = right: $\Phi_{p_0}^{k_n}(J_{p_0}^c)$] at (0, 0.3) {};
            \draw [color = blue] (0,0.3)--(5,3.3);
            \draw [color = blue] (0,1.3)--(5,4.3);

            \draw [thick, color = blue] (8,0.6)-- (8,1.8);
            \node [label = left: $\Phi_{q_0}^{-l_n}(J_{q_0}^c)$] at (8, 0.6) {};
            \draw [color = blue] (8,0.6)--(5,3.6);
            \draw [color = blue] (8,1.8)--(5,4.8);

            \draw [thick, color = blue] (5,3.3)--(5,4.3);
            \draw [thick, color = blue] (5,3.6)--(5,4.8);

            \draw (-0.5,0.5)--(-0.5,2.5);
            \draw (8.625,0.375)--(8.625,2.375);
            \draw (5.125,3.875)--(5.125,5.875);
            
            \draw [dashed, color = blue] (0,0.3)--(-0.5,0.7);
            \draw [dashed, color = blue] (0,1.3)--(-0.5,1.7);
            
            \draw [thick, color = blue] (-0.5,0.7)--(-0.5,1.7);
            \draw [color = blue] (-0.5,0.7)--(5.125,4.075);
            \draw [color = blue] (-0.5,1.7)--(5.125,5.075);
            \node [label = center: $F^{k_n\pi(p_0)}(J_{\widetilde{p}_0}^c)$] at (-1.8,1.5) {};

            \draw [dashed, color = blue] (8,0.6)--(8.625,0.9);
            \draw [dashed, color = blue] (8,1.8)--(8.625,2.1);

            \draw [thick, color = blue] (8.625,0.9)--(8.625,2.1);
            \draw [color = blue] (8.625,0.9)--(5.125,4.4);
            \draw [color = blue] (8.625,2.1)--(5.125,5.6);
            \node [label = center: $F^{-l_n}(J_{\widetilde{q}_0}^c)$] at (9.5,1.5) {};

            \draw [thick, color = blue] (5.125,4.075)--(5.125,5.6);
            
            \node [circle, fill, inner sep = 1pt, label = right: $r_0^n$] at (5.125,3.875) {};
            \node [circle, fill, inner sep = 1pt, label = below: $u_n$] at (4.5,3.5) {};
            \node [circle, fill, inner sep = 1pt, label = below: $v_n$] at (5.625,3.375) {};
        \end{tikzpicture}
        \caption{Dynamics of central intersections}
        \label{figure: central intersection}
    \end{figure}
    
    \begin{claim}\label{claim: center intersection}
        There exist two strictly increasing subsequences $\{k_n\},\{l_n\}$ of $\mathbb{N}$, and a constant $C > 0$, such that
        \[\left|H\circ \Phi_{p_0}^{k_n}(J_{p_0}^c)\cap \Phi_{q_0}^{-l_n}(J_{q_0}^c)\right| \geq C\cdot\min\{e^{\sigma_{p_0}(k_n)},\ e^{-\sigma_{q_0}(l_n)}\},\ \forall\, n\in\mathbb{N},\]
        where
        \[H := \mathrm{Hol}_{r_0}^{q_0}\circ\mathrm{Hol}_{p_0}^{r_0},\]
        and
        \begin{align*}
            &\sigma_{p_0}(k_n) = 
            \left\{\begin{array}{ll}
                 k_n\ln\Phi_{p_0}'(0),& \mathrm{ord}(p_0) = 1; \\[1ex]
                 -\frac{\mathrm{ord}(p_0)}{\mathrm{ord}(p_0) - 1}\ln k_n,& \mathrm{ord}(p_0) \geq 2;
            \end{array}\right.\\
            &\sigma_{q_0}(l_n) = 
            \left\{\begin{array}{ll}
                 l_n\ln \Phi_{q_0}'(0),& \mathrm{ord}(q_0) = 1; \\[1ex]
                 \frac{\mathrm{ord}(q_0)}{\mathrm{ord}(q_0) - 1}\ln l_n,& \mathrm{ord}(q_0) \geq 2.
            \end{array}\right.
        \end{align*}
    \end{claim}
    
    \begin{proof}[Proof of Claim \ref{claim: center intersection}]
        Note that $[\Phi_{p_0}/\Phi_{q_0}^{-1}] = 0$ or $([\Phi_{p_0}/\Phi_{q_0}^{-1}],\ 1) < \varepsilon$. 
        The conclusion then follows from Proposition \ref{proposition: center intersection}, by taking
        \[f = \Phi_{p_0}|_{I_{p_0}};\quad g = \Phi_{q_0}^{-1}|_{I_{q_0}^c};\quad h = H.\]
        The initial intervals are $I = J_{p_0}^c$ and $J = J_{q_0}^c$. 
        We also note that when $g'(0) = 1$, by the choice of $p_0$ and $q_0$, 
        we have that $\mathrm{ord}(p) \geq 2$ for any $p\in\mathrm{Per}(F_0)$. 
        Therefore, by taking a proper coordinate (Claim \ref{claim: Livsic reparameterization}), 
        we have $\phi_x'(0) \equiv 1$ for any $x\in N$. 
        In particular, $h'(0) = 1$.
    \end{proof}

    As a result, there exists a constant $C > 0$ such that
    $$\left|J_{p_0}^n\cap J_{q_0}^n\right|\geq C\cdot\min\left\{e^{\sigma_{p_0}(k_n)},\ e^{-\sigma_{q_0}(l_n)}\right\},\forall\, n\in\mathbb{N},$$
    where
    \begin{align*}
        J_{p_0}^n &:= \mathrm{Hol}_{p_0}^{r_0}\circ \Phi_{p_0}^{k_n}(J_{p_0}^c);\\
        J_{q_0}^n &:= (\mathrm{Hol}_{r_0}^{q_0})^{-1}\circ \Phi_{q_0}^{-l_n}(J_{q_0}^c).
    \end{align*}
    Now consider the intersection between $F^{k_n\pi(p_0)}(\mathcal{F}^u_{loc}(J_{\widetilde{p}_0}^c))$ and $F^{-l_n}(\mathcal{F}^s_{loc}(J_{\widetilde{q}_0}^c))$. Since $F_0$ is an Anosov diffeomorphism on a nilmanifold, for $n$ sufficiently large, their projections on $M_0$, $\pi_0(F^{k_n\pi(p_0)}(\mathcal{F}^u_{loc}(J_{\widetilde{p}_0}^c)))$ and $\pi_0(F^{-l_n}(\mathcal{F}^s_{loc}(J_{\widetilde{q}_0}^c)))$, intersect at $r_0^n\in M_0$, which converges to $r_0$ as $n \to +\infty$. Moreover, we have
    \begin{align*}
        u_n &\in \mathcal{F}^u_{loc}(r_0^n)\pitchfork \mathcal{F}^s_{loc}(r_0),\\
        v_n &\in \mathcal{F}^s_{loc}(r_0^n)\pitchfork \mathcal{F}^u_{loc}(r_0).
    \end{align*}
    For simplicity, we denote the intersections by
    \begin{align*}
        \widehat{J}_{p_0}^n &:= F^{k_n\pi(p_0)}(\mathcal{F}^u_{loc}(J_{\widetilde{p}_0}^c))\cap \mathcal{F}^c(r_0^n);\\
        \widehat{J}_{q_0}^n &:= F^{-l_n}(\mathcal{F}^s_{loc}(J_{\widetilde{q}_0}^c)) \cap \mathcal{F}^c(r_0^n).
    \end{align*}
    Our final goal is to show that $\widehat{J}_{p_0}^n\cap \widehat{J}_{q_0}^n\neq\varnothing$; 
    see Figure \ref{figure: central perturbations}.

    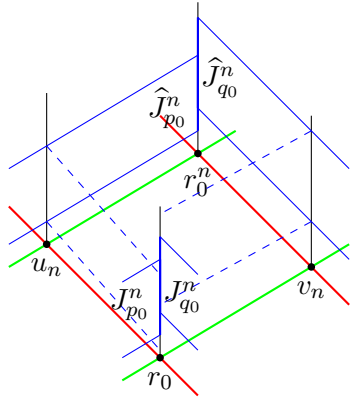
\begin{figure}[htbp]
		\centering
        \begin{tikzpicture}[mid arrow/.style = {decoration = {markings, mark = at position 0.5 with {\arrow{>}}}, postaction = {decorate}}]
            \draw [thick, color = green] (-0.5,-0.3)--(2.5,1.5);
            \draw [thick, color = green] (-2,1.2)--(1,3);
            \draw [thick, color = red] (-2,2)--(0.5,-0.5);
            \draw [thick, color = red] (0,3.2)--(2.5,0.7);
            \node [circle, fill, inner sep = 1pt, label = below: $r_0$] at (0,0) {};
            \node [circle, fill, inner sep = 1pt, label = below: $r_0^n$] at (0.5,2.7) {};
            \node [circle, fill, inner sep = 1pt, label = below: $u_n$] at (-1.5,1.5) {};
            \node [circle, fill, inner sep = 1pt, label = below: $v_n$] at (2,1.2) {};
            \draw (0,0)--(0,2);
            \draw (-1.5,1.5)--(-1.5,3.5);
            \draw (0.5,2.7)--(0.5,4.7);
            \draw (2,1.2)--(2,3.2);

            \draw [thick, color = blue] (0,0.3)--(0,1.3);
            \draw [thick, color = blue] (0,0.6)--(0,1.6);
            \draw [color = blue] (0,0.3)--(-0.5,0);
            \draw [color = blue] (0,1.3)--(-0.5,1);
            \draw [color = blue] (0,0.6)--(0.5,0.1);
            \draw [color = blue] (0,1.6)--(0.5,1.1);
            
            \draw [thick, color = blue] (0.5,3)--(0.5,4);
            \draw [color = blue] (0.5,3)--(-2,1.5);
            \draw [color = blue] (0.5,4)--(-2,2.5);
            \draw [dashed, color = blue] (-1.5,1.8)--(0,0.1);
            \draw [dashed, color = blue] (-1.5,2.8)--(0,1.1);
            
            \draw [thick, color = blue] (0.5,3.3)--(0.5,4.5);
            \draw [color = blue] (0.5,3.3)--(2.5,1.3);
            \draw [color = blue] (0.5,4.5)--(2.5,2.5);
            \draw [dashed, color = blue] (2,1.8)--(0,0.7);
            \draw [dashed, color = blue] (2,3)--(0,1.9);

            \node [label = center: $\widehat{J}_{p_0}^n$] at (0.1,3.3) {};
            \node [label = center: $\widehat{J}_{q_0}^n$] at (0.8,3.7) {};
            \node [label = center: $J_{p_0}^n$] at (-0.4,0.7) {};
            \node [label = center: $J_{q_0}^n$] at (0.3,0.8) {};
        \end{tikzpicture}
        \caption{Dynamics near $r_0$}
        \label{figure: central perturbations}
    \end{figure}
    
    From Lemma \ref{lemma: C1-holonomy} (the $C^1$ property of holonomy maps) 
    and Lemma \ref{lemma: central twist} (the estimate of central twists), 
    there exists a constant $C\geq 1$ such that
    \begin{align*}
        \mathrm{Hol}_{u_n}^{r_0} &\circ \mathrm{Hol}_{r_0^n}^{u_n}(\widehat{J}_{p_0}^n) \text{ is a } Ce^{k_n\pi(p_0)\lambda_{p_0}}\text{-perturbation of } J_{p_0}^n;\\
        \mathrm{Hol}_{v_n}^{r_0} &\circ \mathrm{Hol}_{r_0^n}^{v_n}(\widehat{J}_{q_0}^n) \text{ is a } Ce^{-l_n\lambda_{q_0}} \text{-perturbation of } J_{q_0}^n,
    \end{align*}
    as intervals in $\mathcal{F}^c(r_0)$, where
    \begin{align*}
        \lambda_{p_0} &:= \frac{\lambda^u(p_0) - \lambda^c(p_0)}{\lambda^u(p_0) - \lambda^s(p_0)}\lambda^s(p_0) < \lambda^c(p_0) \leq 0,\\
        \lambda_{q_0} &:= \frac{-\lambda^s(q_0) + \lambda^c(q_0)}{-\lambda^s(q_0) + \lambda^u(q_0)}\lambda^u(q_0) > \lambda^c(q_0) \geq 0.
    \end{align*}
    Again, by the argument of central twist in Lemma \ref{lemma: central twist}, 
    for $n$ sufficiently large and $x\in \widehat{J}_{q_0}^n$, we can get
    \begin{align*}
        &d_c\left(\mathrm{Hol}_{v_n}^{r_0}\circ \mathrm{Hol}_{r_0^n}^{v_n}(x),\ \mathrm{Hol}_{u_n}^{r_0}\circ \mathrm{Hol}_{r_0^n}^{u_n}(x)\right)\\
        \leq& C\left[d_u(r_0^n, u_n) + d_s(u_n, r_0) + d_s(r_0^n, v_n) + d_u(v_n, r_0)\right]\\
        \leq& C^2\left[e^{k_n\pi(p_0)\lambda^s(p_0)} + e^{-l_n\lambda^u(q_0)}\right].
    \end{align*}
    Recall that the length of the intersection is controlled by
    $$\left|J_{p_0}^n\cap J_{q_0}^n\right|\geq C\cdot\min\left\{e^{\sigma_{p_0}(k_n)},\ e^{-\sigma_{q_0}(l_n)}\right\},$$
    and
    \[\sup_n\left|\frac{\sigma_{p_0}(k_n)}{\mathrm{ord}(p_0)} -\frac{-\sigma_{q_0}(l_n)}{\mathrm{ord}(q_0)}\right| < + \infty.\]
    \begin{claim}\label{claim: higher order}
        The decay of perturbations have higher order than the decay of the length of intersection.
    \end{claim}

    \begin{proof}[Proof of Claim \ref{claim: higher order}]
        There are two perturbations: 
        the central twist, $e^{k_n\pi(p_0)\lambda_{p_0}}$, $e^{-l_n\lambda_{q_0}}$; 
        the error from the holonomy, $e^{k_n\pi(p_0)\lambda^s(p_0)}$, $e^{-l_n\lambda^u(q_0)}$. Since
        \[0 > \lambda_{p_0} > \lambda^s(p_0) \quad\text{and}\quad 0 < \lambda_{q_0} < \lambda^u(q_0),\]
        we only need to show that
        \[\max\{e^{k_n\pi(p_0)\lambda_{p_0}}, e^{-l_n\lambda_{q_0}}\} = o\left(\min\{e^{\sigma_{p_0}(k_n)}, e^{-\sigma_{q_0}(l_n)}\}\right).\]
        Notice that
        \[\exp(\sigma_{p_0}(k_n)) \asymp \exp\left(-\frac{\mathrm{ord}(p_0)}{\mathrm{ord}(q_0)}\sigma_{q_0}(l_n)\right).\]
        Since $\mathrm{ord}(p_0) \geq \mathrm{ord}(q_0)$, it suffices to show that
        \[\max\{e^{k_n\pi(p_0)\lambda_{p_0}}, e^{-l_n\lambda_{q_0}}\} = o\left(\exp\left(-\frac{\mathrm{ord}(p_0)}{\mathrm{ord}(q_0)}\sigma_{q_0}(l_n)\right)\right).\]
        There are the following cases.
        
    \noindent \textbf{Case 1}: $\mathrm{ord}(p_0) \geq \mathrm{ord}(q_0) \geq 2$. In this case, we have
        \[k_n^{-\frac{\mathrm{ord}(p_0)}{\mathrm{ord}(p_0) - 1}} \asymp l_n^{-\frac{\mathrm{ord}(p_0)}{\mathrm{ord}(q_0)}\frac{\mathrm{ord}(q_0)}{\mathrm{ord}(q_0) - 1}}.\]
        After taking logarithms, the left-hand side tends to \(-\infty\) at a polynomial rate in \(l_n\), 
        whereas the right-hand side tends to \(-\infty\) at a logarithmic rate.
        
    \noindent \textbf{Case 2}: $\mathrm{ord}(p_0) = \mathrm{ord}(q_0) = 1$. In this case, we have
        \[e^{k_n\pi(p_0)\lambda^c(p_0)} \asymp e^{-l_n\lambda^c(q_0)}.\]
        Therefore, the exponent of the left side is $\max\{\frac{-\lambda^c(q_0)}{\lambda^c(p_0)}\lambda_{p_0}, -\lambda_{q_0}\}$, smaller than the exponent of the right side, $-\lambda^c(q_0)$.

    \noindent \textbf{Case 3}: $\mathrm{ord}(p_0) > \mathrm{ord}(q_0) = 1$. In this case, we have
        \[k_n^{-\frac{\mathrm{ord}(p_0)}{\mathrm{ord}(p_0) - 1}} \asymp \exp\left(-\frac{\mathrm{ord}(p_0)}{\mathrm{ord}(q_0)}\lambda^c(q_0)l_n\right).\]
        Therefore, the exponent of the left side is $-\lambda_{q_0}$, smaller than the exponent of the right side, $-\frac{\mathrm{ord}(p_0)}{\mathrm{ord}(q_0)}\lambda^c(q_0)$, which is ensured by Claim \ref{claim: weak hyperbolicity}.
    \end{proof}
    We then conclude that the perturbations have higher order than the length of intersection. Hence, there exists $n_0\in\mathbb{N}$, such that 
    $\widehat{J}_{p_0}^n\cap \widehat{J}_{q_0}^n\neq\varnothing$, for any $n \geq n_0$. 
    Since $\widehat{J}_{p_0}^n\subseteq U^*$ and $\widehat{J}_{q_0}^n\subseteq V^*$, 
    we conclude that $U^*\cap V^*\neq\varnothing$. 

    Every positive iterate of \(F\) is again an analytic partially hyperbolic skew-product diffeomorphism. 
    Its periodic fiber return maps are positive iterates of those of \(F\), 
    so the central topological sink-type and source-type are preserved. 
    The stable and unstable sets are also unchanged, and hence the weak boundary interconnection is preserved. 
    Applying the transitivity argument above to every positive iterate of \(F\) proves the total transitivity.
    
    This ends the whole proof of Theorem \ref{theorem: main theorem A}.
\end{proof}

\section*{Acknowledgement}
The authors are grateful to Yi Shi for his valuable guidance.
W. Li is partially supported by the NSFC (12571203);
M. Xia is partially supported by the NSFC (12501238) 
and the Fundamental Research Funds for the Central Universities (DUT24RC(3)112). 
The authors used ChatGPT 6, whose use was limited to 
checking computational details 
and refining a few technical points, 
to conduct adversarial review of the preliminary version of this manuscript.

\bibliographystyle{amsalpha}
\bibliography{TAPHI}

\end{CJK}
\end{document}